\documentclass[12pt,leqno,amsfonts,amscd]{amsart}
\usepackage{epsfig}
\usepackage{amsmath}
\usepackage{amsfonts}
\usepackage{amssymb}
\usepackage{hyperref}

\newtheorem{theorem}{Theorem}[section]
\newtheorem{lemma}[theorem]{Lemma}

\newtheorem{cor}[theorem]{Corollary}

\theoremstyle{definition}
\newtheorem{definition}[theorem]{Definition}

\theoremstyle{remark}

\newtheorem{remark}[theorem]{Remark}

\numberwithin{equation}{section}

\newcommand{\NN}{{\mathbb N}}

\newcommand{\RR}{{\mathbb R}}

\newcommand{\eps}{\varepsilon}
\newcommand{\out}[1]{\ }
\DeclareMathOperator{\capa}{cap}

\DeclareMathOperator{\supp}{supp}

\let\cal=\mathcal
\renewcommand{\phi}{\varphi}

\begin{document}

\title[Symmetric function kernels and sweeping of measures]{Symmetric function kernels and sweeping of measures}

\author{Bent Fuglede}
\address{Department of Mathematical Sciences
\\Universitetsparken 5
\\2100 Copenhagen
\\Danmark}
 \email{fuglede@math.ku.dk}


\begin{abstract} This is a potential theoretic study of balayage (sweeping) of a positive Radon measure $\omega$ on a locally compact (Hausdorff) space $X$ onto a closed, or more generally a quasiclosed set $A\subset X$ (that is, a set which can be approximated in outer capacity by closed sets). The setting is that of potentials with respect to a suitable symmetric function kernel $G:X\times X\to[0,+\infty]$. Following Choquet (1959) we consider energy capacity, not as a set function, but as a functional, acting on positive numerical functions on $X$. The finiteness of the upper capacity of the function $1_AG\omega$ is sufficient for the possibility of the sweeping in question ($1_A$ denoting the indicator function of $A$ and $G\omega$ the $G$-potential of $\omega$).
\end{abstract}

\maketitle

\section{Introduction}\label{sec1}
The thesis of Frostman \cite{Fr} marks the beginning of potential theory with respect to other kernels than the Newtonian or Greenian ones. He considered the kernels $|x-y|^{\alpha-n}$ of order $0<\alpha<n$ on $\RR^n$, studied particularly by his teacher M.\ Riesz as published in \cite{MR} (1938). Potential theory with respect to these kernels culminated with the book of Landkof \cite{L}. Potential theory with respect to much more general kernels began around 1940 with many contributions notably from the Japanese school, first by Kametani, Ugaheri, Kunugui, and Ninomiya, and from the French school around Brelot, H. Cartan, Choquet, and Deny. A comprehensive study of the various `principles' in potential theory and their interrelations for rather general kernels was made by Ohtsuka \cite{O} (1961). Fundamental results by Cartan \cite{Ca1} (1945) on Hilbert space aspects of classical potential theory were generalized by the present author \cite{Fu1} (1960) to so-called {\textit{consistent\/}} function kernels on an arbitrary locally compact (Hausdorff) space $X$ (that is, on $X\times X$). A continuation of that, suitable for the study of balayage (`sweeping') of a (positive Radon) measure of finite energy on quasiclosed sets (much as in Cartan \cite{Ca2} for balayage on closed sets), was worked out around 1970 (in the setting of consistent function kernels), and some underlying general aspects were treated in \cite{Fu3} and \cite{Fu4} (1971), but the actual potential theoretic aspects were left unpublished until now. Here a study of the energy capacity as a functional and of balayage on quasiclosed sets with respect to a consistent function kernel is presented. This will be applied in ongoing joint work with Zorii \cite{FZ}. I thank Natalia Zorii for encouraging me to publish the present part of my old material from 1970 (now including balayage of measures of infinite energy, using Choquet \cite{Ch2}) and for going through the entire manuscript thoroughly and constructively.

In the present study we consider a kernel
$G$ on a nonvoid locally compact space $X$, that is, a lower semicontinuous (l.s.c.) function $G:X\times X\to[0,+\infty]$. In the absence of other indication, $G$ shall be symmetric and strictly positive on the diagonal. Further requirements will be listed on the way. We denote by $\cal M^+=\cal M^+(X)$ the cone of all (positive Radon) measures on $X$. The {\textit{potential\/}} $G\mu$ of a measure $\mu\in\cal M^+$ is defined by $G\mu(x)=\int G(x,y)\,d\mu(y)$. Our main purpose is to extend the Gauss variational method, passing first from measures, in particular equilibrium measures, on a compact set $K$ to measures on a {\textit{quasicompact}} set, that is, a set $A\subset X$ such that
\begin{eqnarray}\label{1.1}
\inf\,\bigl\{c^*(A\setminus K):\ K\text{ compact, }K\subset X\bigr\}=0
\end{eqnarray} where $c^*$ denotes the outer energy capacity as a set function defined on all subsets of $X$, cf.\ e.g.\ \cite[Section 2.5]{Fu1}. The main step is to pass from (indicator functions of) compact subsets of $X$ to functions of class $\cal H^*_0$, that is, functions $f:X\to[0,+\infty]$ such that
\[
\inf\,\bigl\{c^*(f-h):\ h\in\cal H_0,\ h\le f\bigr\}=0,
\]
where $\cal H_0$ denotes the cone of all finite upper semicontinuous (u.s.c.) functions $h\ge0$ of compact support in $X$, and where $c^*$ is the extension of the above outer energy capacity of sets to a functional, likewise denoted by $c^*$ and termed the {\textit{upper energy capacity}}, defined on the cone $\cal F^+$ of all functions $f:X\to[0,+\infty]$, cf.\ \cite{Fu4} and eqs.\ (\ref{2.1}) through (\ref{2.3}) below. While this latter step is irrelevant for the study of equilibrium, it becomes very useful for the study in Section 4 (with more assumptions on $G$) of balayage of a measure $\omega$ on $X$ onto a suitable set $A\subset X$ (in the first place: on a quasiclosed set $A$, that is a set like a quasicompact set, but with `compact' replaced by `closed' in (\ref{1.1})). Briefly speaking, this usefulness is because the upper energy capacity $c^*(1_AG\omega)$ (supposed finite) governs the game of balayage.

In Section 2 we study the capacitary measures for a function $f$ from the above class $\cal H^*_0$. These capacitary measures are those measures $\mu$ on $X$ which maximize the Gauss integral $\int(2f-G\nu)d\nu$ as $\nu$ ranges over the cone $\cal E^+$ of (positive) measures of finite energy (cf.\ Theorem \ref{thm2.1}). Their potentials, in particular, possess the following properties:
\begin{eqnarray}\label{1.3}
G\mu\ge f\text{\ q.e.\ on\ $X$ \ and \ $G\mu= f$\ $\mu$-a.e.},
\end{eqnarray}
`q.e.' (quasi-everywhere) meaning: everywhere off some set of zero outer energy capacity (compare with Remark~\ref{key}).

Section 3 deals with a dual notion of energy capacity and corresponding upper and lower dual capacity, denoted $\gamma^*$ and $\gamma_*$ respectively, the former being defined for $f\in\cal F^+$ by
\[
\gamma^*(f)=\inf\,\Bigl\{\Bigl(\int G\lambda\,d\lambda\Bigr)^{1/2}:\ \lambda\in\cal E^+,\ G\lambda\ge f\text{\ q.e.}\Bigr\},
\]
interpreted as $+\infty$ if there is no such measure $\lambda$. This upper dual capacity is a particular case of an `encombrement' in the sense of Choquet \cite{Ch6}. It is easily shown that $\gamma^*(G\omega)=(\int G\omega\,d\omega)^{1/2}$ for any $\omega\in\cal E^+$ (Lemma \ref{lemma3.3}).

We now assume that the kernel $G$ is \textit{consistent\/} and \textit{positive (semi)definite\/}  (Definitions~\ref{def3.5} and \ref{def3.4}). Thus the potential $G\lambda$ of every (positive) measure $\lambda$ of finite energy $\int G\lambda\,d\lambda$ is of class $\cal H^*_0$ and
 every function $f\in\cal F^+$ with $\gamma^*(f)<+\infty$ is majorized q.e.\ by a function of class $\cal H_0^*$ (Lemma \ref{lemma3.8}). Consistency of a strictly positive definite kernel amounts to the cone $\cal E^+$ of (positive) measures of finite energy being {\it complete} in the strong topology on $\cal E^+$ induced by the energy norm topology on the prehilbert space $\cal E$ of all signed measures of finite energy, and such that this strong topology on $\cal E^+$  is finer than the induced vague (that is, weak*) topology, \cite[Section 3.3]{Fu1}. (As observed by H. Cartan \cite{Ca1}, the prehilbert space $\cal E$ is incomplete in the case of the Newtonian kernel on $\RR^n$.)

Under the stated hypotheses, capacity and dual capacity are the same: in particular, $c^*=\gamma^*$ (Corollary \ref{cor3.13}).
For every $f\in\cal H^*_0$ there exist measures $\mu\in\cal E^+$ satisfying (\ref{1.3}) above. Any such measure is called a {\textit{capacitary measure}} for $f$, and the class ${\rm{M}}(f)$ of all these measures $\mu$ is a (nonvoid)
convex subset of $\cal E^+$, compact in the vague topology on $X$ (Theorem \ref{thm3.9}). Of course, if $G$ is strictly positive definite then there is only one capacitary measure for $f$. In any case the upper capacity $c^*$ is sequentially order continuous from below:
$$
c^*(f)=\sup_n\,c^*(f_n)
$$
for any increasing sequence of functions $f_n\in\cal F^+$ (Corollary \ref{cor3.13}). It follows by Choquet's capacitability theorem \cite{Ch5} that every $\cal H_0$-Suslin function $f\in\cal F^+$ is $c$-capacitable (Theorem \ref{thm3.15}). Section 3 closes with a discussion of {\textit{upper capacitary measures}} for an arbitrary function $f\in\cal F^+$ with $c^*(f)<+\infty$, by reducing this to the previous case $f\in\cal H^*_0$.

In Section 4, the above is applied (under additional hypotheses on $X$ and $G$) to establish balayage of any (Radon) measure $\omega\in\cal M^+$ on a quasiclosed set $A$ such that $1_AG\omega\in\cal H^*_0$. This requirement is fulfilled (Lemma \ref{lemma4.0}) if $c^*(1_AG\omega)<+\infty$ and if $G\omega$ is quasicontinuous in the sense that there exists for any $\eps>0$ an open set $V$ with $c(V)<\eps$ such that the restriction of $f$ to $\complement V:=X\setminus V$ is continuous (in the extended sense).

For simplicity of statements we now assume that the consistent and positive definite kernel $G$ is {\textit{strictly positive definite}}, so that $f:=1_AG\omega\in\cal H^*_0$ has just one capacitary measure, which we denote by $\omega^A$.  (Similarly, for a quasicompact set $A$, $f:=1_A$ has just one capacitary  measure (in particular: equilibrium measure in the presence of the maximum principle), denoted $\mu_A$.) If $\omega$ has finite energy then $G\omega\in\cal H^*_0$ by definition of consistency, and hence $G\omega$
is indeed quasicontinuous (\cite[Theorem 2.6]{Fu4}). We must further assume that $G$ satisfies the {\textit{domination principle}} in order to pass from $G\mu\le G\omega$ $\mu$-a.e.\ in (\ref{1.3}) to $G\mu\le G\omega$ everywhere on $X$, where now $\mu$ is $\omega^A$. Then $\omega^A$ has the desired properties of the sweeping of $\omega$ on $A$.

It is desirable to remove the above hypothesis that $\omega\in\cal E^+$. Returning to any  (Radon) measure $\omega$ we establish the quasicontinuity of $G\omega$
by using a result by Choquet \cite{Ch2} about quasicontinuity of $G\omega$ for any Radon measure $\omega$ on $X$ (supposed compact) when $G$ is l.s.c.\ and satisfies the continuity principle of Evans and Vasilesco. As noted in \cite{Ch2} the compactness assumption is easily removed. However, quasicontinuity is understood in \cite{Ch2} with respect to the outer
capacity, $G\text{-}\!\capa^*$ (see  Eq.~(\ref{G-capa}) below), which is smaller than the outer energy capacity (assuming $G$ symmetric), but the two are equal if $G$ satisfies the maximum principle. In this way we obtain our result on balayage of a measure $\omega$ on a quasiclosed set $A$ with $c^*(1_AG\omega)<+\infty$ (Theorem \ref{thm4.5}).

Replacing $G\omega$ by the constant function $1$ (and the hypothesis of the domination principle by that of the maximum principle) this leads to a corresponding result about the equilibrium measure on a quasiclosed set $A\subset X$ of finite upper capacity $c^*(A)$ (Remark \ref{remark4.7}). As emphasized by Natalia Zorii (personal communication), the requirement that $c^*(A)$ be finite is not necessary for the existence of an equilibrium measure on $A$, cf.\ \cite[p.\ 277]{Ca2}, \cite[p.\ 74]{Fu5}, \cite{Z}. This requirement, however, is necessary (and sufficient) for the existence of an equilibrium measure of finite energy, and our method is confined to equilibrium measures (and swept measures $\omega^A$) of finite energy.

Dropping now
the hypothesis that $A$ be quasiclosed in these two results on sweeping, resp.\ equilibrium, we obtain corresponding results on {\textit{upper sweeping}}, resp.\ {\textit{outer equilibrium}}   (Theorem \ref{cor4.8}, resp.\ Corollary \ref{cor4.11}), simply by passing from $A$ to any {\textit{quasiclosure}} $A^*$  of $A$ and noting that $c(1_{A^*}G\omega)=c^*(1_AG\omega)<+\infty$ and hence $1_{A^*}G\omega\in\cal H^*_0$, resp.\ $c(A^*)=c^*(A)<+\infty$, whence $A^*$ is quasicompact. A quasiclosure of a set $A\subset X$ is defined as a quasiclosed set $A^*$ containing $A$ which is minimal with these properties (up to a set of zero outer capacity). For details, see the beginning of Section \ref{OuterBal}.

Dropping instead the requirement that $G$ in Theorem~\ref{thm4.5} (resp.\ Remark \ref{remark4.7}) satisfy the domination principle (resp.\ the maximum principle), we loose the inequality $G\omega_A\leq G\omega$, resp.\ $G\mu_A\leq 1$, everywhere on $X$ and therefore only have `pseudobalayage'  instead of balayage (Theorem \ref{thm4.4}), resp.\ capacitary measures instead of equilibrium measures (Theorem \ref{thm4.6}). These two theorems are valid, e.g., for the Riesz kernels $|x-y|^{\alpha-n}$ on $\RR^n$ of any order $0<\alpha<n$ whereas we have actual balayage and equilibrium for $0<\alpha\leq 2$ only.

\section{An extension of the Gauss variational method}\label{sec2}

\subsection{Definitions and preliminaries.} Let $X$ be a non-void locally compact space. A \textit{(positive function) kernel} on $X$ is a lower semicontinuous (l.s.c.) function $G:X\times X\to[0,+\infty]$. In the absence of other indication we shall furthermore assume throughout (except in Section \ref{relations})
that $G$ is \textit{symmetric}, that is, $G(x,y)=G(y,x)$ for $x,y\in X$, and that $G$ is strictly positive on the diagonal, that is, $G(x,x)>0$ for $x\in X$.

For any (positive Radon) measure $\mu\in\cal M^+=\cal M^+(X)$ define the
 {\textit{potential}}
$G\mu:X\to[0,+\infty]$ and the
{\textit{energy}}
$\int G\mu\,d\mu\in[0,+\infty]$
of $\mu$ by
$$
G\mu(x)=\int G(x,y)\,d\mu(y), \quad \int G\mu\,d\mu=\iint G(x,y)\,d\mu(x)\,d\mu(y).
$$
For brevity we may write $\|\mu\|$ for $(\int G\mu\,d\mu)^{1/2}$. Define
$$
\cal E^+=\bigl\{\mu\in\cal M^+:\ \|\mu\|<+\infty\bigr\},\quad\cal E^+_1=\bigl\{\mu\in\cal M^+:\ \|\mu\|\le1\bigr\} .
$$

 \begin{remark}\label{remark2.0} The requirement that $G$ be strictly positive on the diagonal is equivalent to $\cal E^+_1$ being vaguely bounded (or equivalently: vaguely compact), see \cite[Lemma 2.5.1]{Fu1}. It follows that $\|\mu\|>0$ for any non-zero $\mu\in\cal M^+$. To see this, take $z\in\supp\mu$ and an open neighborhood $V$ of $z$ such that $G>0$ on $V\times V$ (possible since $G(z,z)>0$ and $G$ is l.s.c.). Then $\|\mu\|^2\ge\int_V\int_VG(x,y)\,d\mu(x)\,d\mu(y)>0$.
\end{remark}

The {\textit{mutual energy}}
of two measures $\mu,\nu\in\cal M^+$ is defined as
$$
\int G\mu\,d\nu=\int G\nu\,d\mu=\iint G(x,y)\,d\mu(x)\,d\nu(y)\in[0,+\infty].
$$
The potential, the energy, and the mutual energy are l.s.c.\ functions of the respective  variables, $\cal M^+$ being given the vague topology, \cite[Lemma 2.2.1]{Fu1}.

Let $\cal F^+=\cal F^+(X)$ denote the convex cone of all functions $X\to[0,+\infty]$. Let $\cal G$ denote the convex subcone of $\cal F^+$ consisting of all l.s.c.\ functions $X\to[0,+\infty]$. Let $\cal H_0$ denote the convex subcone of $\cal F^+$ consisting of all finite and upper semicontinuous (u.s.c.) functions $X\to[0,+\infty[\,$ of compact support in $X$. Then  $\cal C^+_0:=\cal G\cap\cal H_0$ is the further convex subcone of $\cal F^+$ formed by all finite continuous positive functions of compact support in $X$. Referring to \cite[Section 5]{Fu4} we shall consider the enveloping capacity for the set $\cal E^+_1$ of measures. This capacity $c$ is called the \textit{energy capacity} with respect to the kernel $G$. Explicitly, the functional $c:\cal C^+_0\to[0,+\infty]$ is defined by
\begin{eqnarray}
c(\phi)=\max_{\mu\in\cal E^+_1}\,\mu(\phi)
=\max\,\Bigl\{\int\phi\,d\mu:\ \mu\in
\cal E^+,\ \int G\mu\,d\mu\le1\Bigr\},\label{2.1}
\end{eqnarray}
$\phi\in\cal C^+_0$. If we want to specify the kernel $G$ in question we may write $c_G$ for $c$. As a functional on $\cal C^+_0$, the energy capacity $c$ is finite valued, increasing, and sublinear (that is, subadditive and positive homogeneous), \cite[Section 5.4]{Fu4},
applicable because $G$ is strictly positive on the diagonal (cf.\ Remark \ref{remark2.0}).

The extensions of this functional to functions of class $\cal H_0$ or $\cal G$ are defined \cite[Section 4.2]{Fu4} by
\begin{align*}
c(h)&=\inf\,\bigl\{c(\phi):\ \phi\in\cal C^+_0,\ \phi\ge h\bigr\},\quad h\in\cal H_0,\\
c(g)&=\sup\,\bigl\{c(\phi):\ \phi\in\cal C^+_0,\ \phi\le g\bigr\},\quad g\in\cal G.
\end{align*}
The expression (\ref{2.1}) for $c(\phi)$ with $\phi\in\cal C^+_0$ remains valid with $\phi$ replaced by $h\in\cal H_0$, \cite[Theorem 5.5]{Fu4}. See also Theorem \ref{thm2.0} below.

Define the \textit{lower} and the \textit{upper (energy) capacity} of a function $f\in\cal F^+$ by
\begin{align}
c_*(f)&=\sup\,\bigl\{c(h):\ h\in\cal H_0,\ h\le f\bigr\}=\sup\,\Bigl\{\int_*f\,d\mu:\ \mu\in\cal E^+_1\Bigr\},\label{2.2}\\
c^*(f)&=\inf\,\bigl\{c(g):\ g\in\cal G,\ g\ge f\bigr\}\ge\sup\,\Bigl\{\int^*f\,d\mu:\ \mu\in\cal E^+_1\Bigr\},\label{2.3}
\end{align}
cf.\ \cite[Section 5.5]{Fu4} for the latter relation in (\ref{2.2}) and (\ref{2.3}). In the latter equality (\ref{2.2}) it suffices to admit measures $\mu\in\cal E^+_1$ of compact support contained in $\{f>0\}$, \cite[Section 5.5, Remark]{Fu4}. These two functionals $c_*$ and $c^*$ on $\cal F^+$ are increasing and positive homogeneous; and $c^*$ is \textit{countably subadditive}, \cite[p.\ 21]{Fu4}. Note that $c^*$ is an \textit{upper capacity} also in the sense of \cite[Definition 3.1]{Fu4}. Two functions $f_1,f_2\in\cal F^+$ are said to be \textit{$c^*$-equivalent} if $c^*(|f_1-f_2|)=0$, or equivalently if $f_1=f_2$ q.e., cf.\ \cite[p.~6, Corollary 2]{Fu4}. (It is understood that $|f_1-f_2|=+\infty$ at points where $f_1=f_2=+\infty$.)

We say that $f$ is $c$\textit{-capacitable} if $c^*(f)=c_*(f)$, in which case we may write $c(f)$ in place of $c^*(f)$ or $c_*(f)$, and briefly term $c(f)$ the \textit{(energy) capacity} of $f$. Define $\cal H_0^*$ and $\cal G^*$ as the closures of $\cal H_0$ and $\cal G$, respectively, in the $c^*$-metric topology, that is, the topology on $\cal F^+$ defined by the pseudometric (\'ecart) $(f_1,f_2)\mapsto c^*(|f_1-f_2|)$. (We always define $(+\infty)-(+\infty)=+\infty$.) Equivalently, for $f\in\cal F^+$,
\begin{eqnarray*}f\in\cal H_0^*
&\Longleftrightarrow&\inf\,\bigl\{c^*(f-h):\ h\in\cal H_0,\ h\le f\bigr\}=0,\\
f\in\cal G^*
&\Longleftrightarrow&\inf\,\bigl\{c^*(g-f):\ g\in\cal G,\ g\ge f\bigr\}=0,
\end{eqnarray*}
by \cite[Section 3.2]{Fu4}.

Every function of class $\cal H_0$ or $\cal G$ is $c$-capacitable, and so is therefore every function of class $\cal H^*_0$ or $\cal G^*$, see \cite[Lemma 4.6]{Fu4}. Furthermore, $c(h)<+\infty$ for any $h\in\cal H_0$ and hence for any $h\in\cal H^*_0$; this is because any $h\in\cal H_0$ is majorized by some $\phi\in\cal C_0^+$, and hence $c(h)\le c(\phi)<+\infty$. Also, every function of class $\cal G^*$ is measurable and every function of class $\cal H^*_0$ is integrable with respect to any $\mu\in\cal E^+$, \cite[Corollary 6.1]{Fu4}.

The following further extension of (\ref{2.1}) from functions of class $\cal H_0$, now to functions of class  $\cal H^*_0$, is crucial for the present study. It is a particular case of \cite[Theorem 6.3]{Fu4}.

\begin{theorem}\label{thm2.0} For any function $f\in\cal H^*_0$
$$
c(f)=\max_{\nu\in\cal E^+_1}\,\nu(f)=\max\,\Bigl\{\int f\,d\nu:\ \nu\in\cal E^+,\ \int G\nu\,d\nu\le1\Bigr\}.
$$
\end{theorem}

If $c(f)>0$ then every maximizing measure $\nu$ for $c(f)$ in Theorem \ref{thm2.0} clearly has energy 1. The expression for $c(f)$ in the theorem therefore remains valid if $\cal E^+_1$ is replaced by its boundary $\partial\cal E^+_1=\{\nu\in\cal E^+:\ \|\nu\|=1\}$.

Identifying a set $A\subset X$ with its indicator function $1_A\in\cal F^+$ we write $c(K)=c(1_K)$ for compact $K$. Then $c(K)^2$ is the usual energy capacity of $K$, cf.\ e.g.\ \cite[Eq.\ (1) and note 1, p.\ 162]{Fu1}. Denoting by $c_*(A)^2$ and $c^*(A)^2$ the usual inner and outer energy capacity of an arbitrary set $A$ it follows that $c_*(A)=c_*(1_A)$ and $c^*(A)=c^*(1_A)$. We have the usual notions nearly everywhere (n.e.) and quasi-everywhere (n.e.), that is, everywhere except in some set $E$ with $c_*(E)=0$, resp.\ $c^*(E)=0$. According to \cite[Lemma 2.4]{Fu4} the sets $A$ such that $1_A\in\cal H^*_0$ are the quasicompact sets, see (\ref{1.1}).

\subsection{Gauss variation for a function $f\in\cal H^*_0$.}
The key to the rest of this paper is the extension of the Gauss variational method from the classical case dealing with compact sets $K\subset X$ (see e.g.\ \cite[Theorem 2.5]{Fu1}), to dealing with functions of class $\cal H_0$ or even $\cal H_0^*$, in particular with (indicator functions of) quasicompact sets.

The \textit{Gauss integral} associated with a given function $f\in\cal H^*_0$ is defined as the following function of a variable measure $\nu\in\cal E^+$:
$$
\int(2f-G\nu)\,d\nu=2\int f\,d\nu-\|\nu\|^2,
$$
which is finite and only depends on the $c^*$-equivalence class of $f\in\cal H^*_0$ because $\nu$ does not charge the sets of zero outer capacity.

We say that a measure $\mu\in\cal M^+$ is \textit{carried by\/} (or \textit{concentrated on\/}) a set $A\subset X$ if $\complement A$ is locally $\mu$-negligible. Let $\cal E^+(A)$ consist of all $\mu\in\cal E^+$ that are carried by~$A$.

\begin{theorem}\label{thm2.1} For any function $f\in\cal H^*_0$ we have
\begin{align}c(f)^2
&=\max_{\nu\in\cal E^+}\,\int(2f-G\nu)\,d\nu=\max_{\nu\in\cal E^+}\,\biggl\{\int2f\,d\nu-\|\nu\|^2\biggr\}<+\infty,\label{2.4}\\
c(f)^2
&=\max\,\biggl\{\int f\,d\nu:\ \nu\in\cal E^+,\ G\nu\le f{\text{ $\nu$-a.e.}}\biggr\}.
\label{2.5}
\end{align}
The maximizing measures are the same in the two cases\/ {\rm(\ref{2.4})} and\/ {\rm(\ref{2.5})}. They form a (nonvoid) vaguely compact subclass ${\rm{M}}(f)$ of $\cal E^+$ which only depends on the $c^*$-equivalence class of $f$. Each measure $\mu\in{\rm{M}}(f)$ is carried by $\{f>0\}$ and has the following properties:
\begin{itemize}
\item[\rm(a)] $G\mu\ge f{\text{\ q.e.}}$,
\item[\rm(b)] $G\mu=f\text{\ $\mu$-a.e.}$,
\item[\rm(c)] $\int f\,d\mu=\|\mu\|^2=c(f)^2$.
\end{itemize}
\end{theorem}

The measures $\mu\in{\rm{M}(f)}$ are called the {\textit{capacitary measures}} for $f$.

\begin{proof} Let first $\mathrm M(f)$ consist of all $\mu\in\cal E^+$ maximizing the Gauss integral (over $\cal E^+$), provided such maximizing measures exist. In proving (\ref{2.4}) and the equations (c), the latter for $\mu\in\mathrm M(f)$, we may leave out the trivial case $c(f)=0$ in which $\int f\,d\nu=0$ for every $\nu\in\cal E^+$, by the latter relation in (\ref{2.2}) or (\ref{2.3}), or by Theorem \ref{thm2.0}, and so the Gauss integral equals $-\|\nu\|^2$, whose maximum over $\cal E^+$ is $0$. Observe that then $\mathrm M(f)=\{0\}$, for if $\mu\ne0$, then $\|\mu\|>0$ by Remark~\ref{remark2.0}. We are thus left with the case $c(f)>0$, where we only need to consider non-zero competing measures $\nu$ in (\ref{2.4}). By normalization write $\nu=t\nu_1$ with $t=\|\nu\|$, $\nu_1\in\partial\cal E^+_1$. The Gauss integral at $\nu\in\cal E^+$ then becomes
\begin{eqnarray}\label{2.6}
2t\int f\,d\nu_1-t^2=\Bigl(\int f\,d\nu_1\Bigr)^2-\Bigl(t-\int f\,d\nu_1\Bigr)^2.
\end{eqnarray}
For variable $t\in \,]0,+\infty[\,$ and fixed $\nu_1\in\partial\cal E^+_1$ this expression attains its greatest value $(\int f\,d\nu_1)^2$ at $t=\int f\,d\nu_1$. For the corresponding measure $\nu=t\nu_1$ we have
$$
\int f\,d\nu=t^2=\|\nu\|^2.
$$
When now varying $\nu_1$ in $\partial\cal E^+_1$ and maximizing $(\int f\,d\nu_1)^2$, this leads to the greatest value of the Gauss integral (over $\cal E^+$), and that greatest value equals $c(f)^2$  by  Theorem \ref{thm2.0}. Relation (\ref{2.4}) has thus been completely proved. At the same time we have established (c) for every $\mu\in{\rm{M}}(f)$.

Denote by  ${\rm{M}}_1(f)$ the class of all measures $\mu_1\in\cal E^+_1$ (equivalently: $\mu_1\in\partial\cal E^+_1$) which maximize $\int f\,d\nu_1$ over $\cal E^+_1$. We show that
\begin{eqnarray}\label{2.7}
{\rm{M}}(f)=c(f){\rm{M}}_1(f).
\end{eqnarray}
Again we may clearly assume that $c(f)>0$. If $\mu\in{\rm{M}}(f)$ then $\mu_1:=t^{-1}\mu\in{\rm{M}}_1(f)$ for $t:=\|\mu\|=c(f)$. In fact, $\mu_1\in\cal E^+_1$ and $c(f)^2=(\int f\,d\mu_1)^2=\|\mu\|^2$ by (c), so that $\mu_1$ indeed belongs to ${\rm{M}}_1(f)$ according to Theorem~\ref{thm2.0}. Conversely, if $\mu_1\in{\rm{M}}_1(f)$ then $\mu:=t\mu_1\in{\rm{M}}(f)$ for $t:=\int f\,d\mu_1=c(f)$. In fact, $\mu\in\cal E^+$ and the Gauss integral at $\mu$ equals  $(\int f\,d\mu_1)^2=c(f)^2$, by (\ref{2.6}).

It follows from (\ref{2.7}) that ${\rm{M}}(f)$ is vaguely compact (even if $c(f)=0$ and hence $\textrm{M}(f)=\{0\}$), and that every measure $\mu\in{\rm{M}}(f)$ is carried by $\{f>0\}$, because ${\rm{M}}_1(f)$ has these properties.
For the vague compactness of ${\rm{M}}_1(f)$, note that the function $\nu\mapsto\nu(f)$ on $\cal E^+_1$ is vaguely u.s.c.\ by \cite[Theorem 6.2]{Fu4}, whence ${\rm{M}}_1(f)=\bigl\{\nu\in\cal E^+_1:\ \nu(f)\ge c(f)\bigr\}$ is vaguely closed, and therefore vaguely compact along with $\cal E^+_1$ (Remark~\ref{remark2.0}). And every $\mu\in {\rm{M}}_1(f)$ is carried by
$\{f>0\}$ by Theorem \ref{thm2.0} and \cite[Lemma 6.5]{Fu4} because $\cal E^+_1$ is {\textit{strictly hereditary}} in the sense that, for every $\mu\in\cal E^+_1$ and every $\nu\in\cal M^+\setminus\{\mu\}$ with $\nu\le\mu$ we have $\int G\nu\,d\nu<1$. In fact,
$$
\int G\mu\,d\mu\ge\int G\nu\,d\nu+\int G(\mu-\nu)\,d(\mu-\nu)>\int G\nu\,d\nu
$$
by Remark \ref{remark2.0}.

To establish (a) and (b) (also if $c(f)=0$), consider any measure $\mu\in{\rm{M}}(f)$. For any $\nu\in\cal E^+$ we have $\mu+t\nu\in\cal M^+$ for all $t\in\,]0,+\infty[$, and hence
$$
2\int f\,d(\mu+t\nu)-\|\mu+t\nu\|^2\le2\int f\,d\mu-\|\mu\|^2,
$$
also if $\|\mu+t\nu\|=+\infty$, for $\int f\,d\mu$ and $\int f\,d\nu$ are finite. Thus,
$$
2t\int(f-G\mu)\,d\nu-t^2\|\nu\|^2\le0,
$$
and therefore
\begin{eqnarray*}
\int f\,d\nu\le\int G\mu\,d\nu
\end{eqnarray*}
for all $\nu\in\cal E^+$. According to \cite[Lemma 2.3]{Fu4} $(f-G\mu)^+$ is of class $\cal H^*_0$ and hence $\nu$-integrable for every $\nu\in\cal E^+$ and  $c$-capacitable. Since $\int(f-G\mu)\,d\nu\le0$ for any $\nu\in\cal E^+$ as shown above, we thus have $\int(f-G\mu)^+\,d\nu=0$ for any $\nu\in\cal E^+(\{(f-G\mu)^+>0\})$, and clearly even for arbitrary $\nu\in\cal E^+$. Hence $c^*((f-G\mu)^+)= c_*((f-G\mu)^+)=0$, by the latter expression in (\ref{2.2}), and we conclude by \cite[Lemma 1.3 (b)]{Fu4} that indeed $f\le G\mu$ q.e. Having thus established (a) we get in particular
\begin{eqnarray*}\label{2.8}
G\mu\ge f{\text{\ $\nu$-a.e.\ for every\ $\nu\in{\cal E}^+$}}
\end{eqnarray*}
because $\nu^*(\{G\mu<f\})=0$ on account of the latter inequality (\ref{2.3}). Hence, $G\mu\ge f$ $\mu$-a.e. Since, by (c), $\int G\mu\,d\mu=\int f\,d\mu<+\infty$, we thus get $G\mu=f$ $\mu$-a.e., which proves (b).

It remains to establish (\ref{2.5}) and the fact that the maximizing measures in (\ref{2.5}) form the class $\mathrm M(f)$. For any competing measure $\nu$ in (\ref{2.5}) we find by integration with respect to $\nu$ that $\|\nu\|^2\le\int f\,d\nu$ and hence
\begin{eqnarray}\label{2.9}
\int f\,d\nu\le2\int f\,d\nu-\|\nu\|^2\le c(f)^2.
\end{eqnarray}
On the other hand, every $\mu\in{\rm{M}}(f)$ is a competing measure in (\ref{2.5}) by (b) and gives the value $c(f)^2$ to $\int f\,d\mu$ by (c). On account of (\ref{2.9}), we thus see that (\ref{2.5}) holds and every $\mu\in{\rm{M}}(f)$ is maximizing in (\ref{2.5}). Conversely, every maximizing measure $\mu$ for (\ref{2.5}) has $\int G\mu\,d\mu\le\int f\,d\mu$ and hence gives the value $c(f)^2$ to the Gauss integral:
$$
2\int f\,d\mu-\int G\mu\,d\mu\ge\int f\,d\mu=c(f)^2,
$$
whence ${\mu\in\rm{M}}(f)$.
\end{proof}

\begin{remark}\label{key} The properties (a)--(c) characterize $\mu\in\mathrm M(f) $ uniquely up to $c^*$-equivalence. Namely, if $\mu\in\mathcal E^+$ possess (a)--(c), then $\mu\in\mathrm M(f)$.
In the above theorem, property (c) alone characterizes $\mathrm M(f)$ within $\mathcal E^+$. For if $\mu\in\mathcal E^+$ has property (c) then $\mu\in\mathrm M(f)$ because
$$
2\int f\,d\mu-\|\mu\|^2=\|\mu\|^2=c(f)^2,
$$
so that $\mu$ maximizes the Gauss integral. Judging from the classical theory one might expect that the key properties (a) and (b) together might characterize $\mathrm M(f)$, but this does not seem to be the case in the general setting of Theorem \ref{thm2.1}. However, if $G$ is \textit{consistent} (Definition \ref{def3.5}) and \textit{positive (semi)definite} (Definition \ref{def3.4}) then (a) and (b) together imply that $\mu\in\mathrm M(f)$, see Theorem \ref{thm3.9}.
\end{remark}

\section{The dual energy capacity}\label{sec3}

\subsection{Upper and lower dual energy capacity.} We shall study a kind of {\textit{dual}} energy capacity $\gamma$ and corresponding upper and lower dual capacity $\gamma^*,\gamma_*:\ \cal F^+\to[0,+\infty]$. This concept is a particular case of an `encombrement' in the sense of Choquet \cite{Ch6}. It will be shown in Corollary \ref{cor3.13} that $\gamma^*=c^*$ and $\gamma_*=c_*$ if the kernel $G$ is consistent and positive (semi)definite (see the next two subsections for these concepts). To begin with, $G$ is just required to be symmetric, l.s.c., and strictly positive on the diagonal, as stated in the beginning of Section 2.

With any function $f\in\cal F^+$, that is, $f:X\to[0,+\infty]$, we associate the following two convex subsets of $\cal E^+$:
\begin{align*}\Gamma^*(f)
:&=\bigl\{\lambda\in\cal E^+:\ c^*((f-G\lambda)^+)=0\bigr\}\\
&=\bigl\{\lambda\in\cal E^+:\ G\lambda\ge f\text{\ q.e.}\bigr\},\\
\Gamma_*(f):&=\bigl\{\lambda\in\cal E^+:\ c_*((f-G\lambda)^+)=0\bigr\}\\
&=\Bigl\{\lambda\in\cal E^+:\ \int G\lambda\,d\nu\ge\int_*f\,d\nu\text{\ for all\ $\nu\in\cal E^+$}\Bigr\}\\
&(=)\bigl\{\lambda\in\cal E^+:\ G\lambda\ge f\text{\ n.e.}\bigr\}.
\end{align*}
Here $(=)$ indicates equality provided that $f$ is $\nu$-measurable for every $\nu\in\cal E^+$ of compact support contained in $\{f>0\}$; however, if $f\in\cal F^+$ is \textit {arbitrary\/}, then the relation $(=)$ should be replaced by ${}\supset{}$.
These and the alternative expression for $\Gamma^*(f)$ are immediate by \cite[Sections~1.3 and~7.4]{Fu4} because $G\lambda$ is l.s.c.\ and hence $\cal E^+$-measurable (that is, measurable with respect to every measure in $\cal E^+$). For the second representation of $\Gamma_*(f)$ we use the latter expression (\ref{2.2}).

Clearly, $\Gamma^*(f)\subset\Gamma_*(f)$. For $f_1,f_2\in\cal F^+$ the relations
\begin{eqnarray}\label{3.1}
\Gamma^*(f_1)\supset\Gamma^*(f_2),\quad\Gamma_*(f_1)\supset\Gamma_*(f_2)
\end{eqnarray}
hold when $f_1\le f_2$ q.e. Hence $\Gamma^*(f)$ and $\Gamma_*(f)$ only depend on the $c^*$-equivalence class of $f\in\cal F^+$. The latter relation (\ref{3.1}) likewise holds under the weaker hypothesis that $\int_* f_1\,d\nu\le\int_* f_2\,d\nu$ for every $\nu\in\cal E^+$. It follows that, if $f_2$ is $\cal E^+$-measurable, then the latter relation (\ref{3.1}) holds if $c_*((f_1-f_2)^+)=0$, thus in particular if $f_1\le f_2$ n.e., cf.\ again \cite[Theorem 7.4]{Fu4}.

\begin{definition}\label{def3.1} The {\textit{upper}} and {\textit{lower}} dual capacity of a function $f\in\cal F^+$ are defined by
$$
\gamma^*(f)=\inf\,\bigl\{\|\lambda\|:\ \lambda\in\Gamma^*(f)\bigr\},\quad
\gamma_*(f)=\inf\,\bigl\{\|\lambda\|:\ \lambda\in\Gamma_*(f)\bigr\},
$$
respectively, interpreted as $+\infty$ if $\Gamma^*(f)$, resp.\ $\Gamma_*(f)$, is void.
\end{definition}

The terms upper, resp.\ lower, are justified here by Theorem \ref{thm3.12} below (where $G$ is assumed to be consistent in the sense of Definition \ref{def3.5}), resp.\ consistent and positive (semi)definite, cf.\ Definition \ref{def3.4}.

The values $\gamma^*(f)$ and $\gamma_*(f)$ only depend on the $c^*$-equivalence class of $f$. Since $\Gamma^*(f)\subset\Gamma_*(f)$ we have $\gamma_*(f)\le\gamma^*(f)$. In case of equality we may denote the common value by $\gamma(f)$, the {\textit{dual capacity}} of $f$.

Either functional $\gamma^*(f)$ or $\gamma_*(f)$ is positive homogeneous, and if $f_1\le f_2$ q.e.\ then $\gamma^*(f_1)\le\gamma^*(f_2)$ and $\gamma_*(f_1)\le\gamma_*(f_2)$. When $f_2$ is $\cal E^+$-measurable the latter inequality holds if just $c_*((f_1-f_2)^+)=0$, in particular if $f_1\le f_2$ n.e., cf.\ the comments to (\ref{3.1}) above.

\begin{lemma}\label{lemma3.2} Let $f\in\cal H^*_0$. Then $\Gamma^*(f)=\Gamma_*(f)$ and hence $\gamma^*(f)=\gamma_*(f)\;(=\gamma(f))$. Furthermore, $\gamma(f)\le c(f)$.
\end{lemma}

\begin{proof} From $f\in\cal H_0^*$ follows $(f-G\lambda)^+\in\cal H_0^*$ for any $\lambda\in\cal M^+$ by \cite[Lemma 2.3]{Fu4} because $G\lambda\in\cal G\subset\cal G^*$. Hence $c^*((f-G\lambda)^+)=c_*((f-G\lambda)^+)$, and so $\Gamma^*(f)=\Gamma_*(f)$ (then also denoted $\Gamma(f)$). By Theorem \ref{thm2.1} (a), a capacitary measure $\mu\in{\rm{M}}(f)$ belongs to $\Gamma^*(f)$, which in view of  Theorem \ref{thm2.1} (c) yields $\gamma^*(f)\le\|\mu\|=c(f)$.
\end{proof}

\begin{lemma}\label{lemma3.3} For any measure $\mu\in\cal E^+$ we have $\gamma^*(G\mu)=\gamma_*(G\mu)=\|\mu\|$.
\end{lemma}

\begin{proof} Since $\mu\in\Gamma^*(G\mu)$ we have $\gamma^*(G\mu)\le\|\mu\|$. To prove that $\|\mu\|\le\gamma_*(G\mu)$ note that, for any $\lambda\in\Gamma_*(G\mu)$, we have $\int G\lambda\,d\nu\ge\int G\mu\,d\nu$ for all $\nu\in\cal E^+$, and hence
$$
\|\mu\|^2=\int G\mu\,d\mu\le\int G\lambda\,d\mu=\int G\mu\,d\lambda
\le\int G\lambda\,d\lambda=\|\lambda\|^2.
$$
\vskip-3mm
\end{proof}

\subsection{Consistent kernel.}
\begin{definition}\label{def3.5} A (symmetric) kernel $G:X\times X\to[0,+\infty]$ is said to be \textit{consistent} if, for every measure $\lambda\in\cal E^+$, the following two equivalent conditions are fulfilled:
\begin{itemize}
\item[\rm(i)] $G\lambda\in\cal H^*_0$,
\item[\rm(ii)] the vaguely l.s.c.\ function $\mu\mapsto\int G\lambda\,d\mu$ considered on $\cal E^+_1$ is finite and vaguely u.s.c.\ (and hence vaguely continuous).
\end{itemize}
\end{definition}

Since $G\lambda\in\cal G$ the equivalence of (i) and (ii) follows from \cite[Theorem 6.2]{Fu4}. Property (ii) is the same as property (CW) in \cite[Lemma 3.4.1]{Fu1}, where $G$ is supposed to be positive (semi)definite, cf.\ Definition \ref{def3.4} below. And for any positive definite (symmetric) kernel $G$, property (CW) is equivalent with the property (C) of consistency defined in  \cite[p.\ 167]{Fu1}, cf.\ \cite{Fu2}.

\begin{remark}\label{remark3.8a} Suppose that $G$ is consistent. For any function $g\in\cal G^*$
we have
$\gamma^*(g)=\gamma_*(g)$ $({}=\gamma(g))$ because
$(g-G\mu)^+$
is of class $\cal G^*$ and hence $c$-capacitable for any measure $\mu\in\cal E^+$ by \cite[Lemma 2.3]{Fu4} (note that $G\mu\in\cal H^*_0$ by Definition \ref{def3.5} (i)).
We then have $\|\mu\|=\gamma(G\mu)\le c(G\mu)$. The equality holds by Lemma \ref{lemma3.3}. The inequality holds by Lemma \ref{lemma3.2} in view of Definition \ref{def3.5} (i); and equality prevails here if and only if $G$ in addition is \textit{positive (semi)definite} (Definition \ref{def3.4}).
\end{remark}

\begin{theorem}\label{thm3.6} {\rm(Convergence theorem.)} Suppose that $G$ is consistent. For any vague cluster point $\lambda\in\cal E^+$ for a sequence of measures $\lambda_n\in\cal E^+$ with bounded energy norms $\|\lambda_n\|\leqslant a<+\infty$ for some finite constant $a$ we have
$$
G\lambda\ge\liminf_n\,G\lambda_n\text{\ q.e.}
$$
Equality prevails q.e.\  in case $\lambda_n\to\lambda$ vaguely as $n\to\infty$.
\end{theorem}

\begin{proof} Such a cluster point exists by Remark \ref{remark2.0}. For any $p,q\in\NN$ write
\begin{align*}
N&=\bigl\{G\lambda<\liminf_n\,G\lambda_n\bigr\},\\
N_p&=\bigl\{G\lambda+1/p<\liminf_n\,G\lambda_n\bigr\},\\
N_{p,q}&=\bigl\{G\lambda+1/p\le\inf_{n\ge q}\,G\lambda_n\bigr\}.
\end{align*}
Then
$$
N=\bigcup_p\,N_p,\quad N_p\subset\bigcup_q\,N_{p,q},
$$
and in view of the countable subadditivity of $c^*$ it therefore suffices to prove that $c^*(N_{p,q})=0$ for any $p,q$. By Definition \ref{def3.5}, all $G\lambda_n$ are of class $\cal H^*_0$, and so is therefore $\inf_{n\ge q}G\lambda_n$ for each $q$ by \cite[Theorem 2.2]{Fu4}. In particular, $\inf_{n\ge q}G\lambda_n$ is quasi u.s.c., by \cite[Theorem 2.5 (b)]{Fu4}. (For the notion of quasi u.s.c.\ functions, see \cite[Section 3]{Fu3} or Section \ref{quasi} below.) Since $G\lambda+1/p$ is l.s.c.\ it follows that each $N_{p,q}$ is quasiclosed, and in fact quasicompact, being a subset of $A:=\{G\lambda_q\ge1/p\}$ which is quasicompact according to \cite[Lemma 2.4]{Fu4} because $1_A\le pG\lambda_q$, hence $1_A\in\cal H^*_0$, again by \cite[Theorem 2.5 (b)]{Fu4}. Being quasicompact, $N_{p,q}$ is $c$-capacitable with finite $c$-capacity, cf.\ the version of the proof of \cite[Lemma 4.6]{Fu4} for sets instead of functions. Thus it remains to show that $c_*(N_{p,q})=0$. Let $\nu\in\cal E^+$  have compact support contained in $N_{p,q}$ (cf.\ the text following (\ref{2.2}) and (\ref{2.3})). We conclude that $\nu=0$ as follows:
\begin{align*} \int(G\lambda+1/p)\,d\nu
&\le\int(\inf_{n\ge q}G\lambda_n)\,d\nu\le\inf_{n\ge q}\,\int G\lambda_n\,d\nu=\inf_{n\ge q}\,\int G\nu\,d\lambda_n\\
&\le\liminf_n\,\int G\nu\,d\lambda_n\le\int G\nu\,d\lambda=\int G\lambda\,d\nu,
\end{align*}
where the last inequality holds since $G$ is consistent, cf.\ (ii) in Definition \ref{def3.5} according to which the function $\lambda\mapsto\int G\nu\,d\lambda$ on $a\cal E^+_1$ is vaguely continuous, and since $\lambda$ is a vague cluster point of $(\lambda_n)\subset a\cal E^+_1$, $\int G\nu\,d\lambda$ becomes a cluster point of $(\int G\nu\,d\lambda_n)$. This proves the former assertion of the theorem. The latter assertion follows in view of the lower semicontinuity of the kernel $G$.
 \end{proof}

\begin{theorem}\label{thm3.7} Suppose that $G$ is consistent. For any $f\in\cal F^+$ such that $\gamma^*(f)$, resp.\ $\gamma_*(f)$, is finite, the infima in Definition\/ {\rm\ref{def3.1}} are attained:
$$
\gamma^*(f)=\min\,\bigl\{\|\lambda\|:\ \lambda\in\Gamma^*(f)\bigr\},\quad
\gamma_*(f)=\min\,\bigl\{\|\lambda\|:\ \lambda\in\Gamma_*(f)\bigr\}.
$$
\end{theorem}

\begin{proof} For the case of $\gamma^*(f)$ choose a sequence of measures $\lambda_n\in\Gamma^*(f)$ so that $\|\lambda_n\|\to\gamma^*(f)$, and denote by $\lambda$ a vague cluster point for the vaguely bounded sequence $(\lambda_n)$, cf.\ Remark \ref{remark2.0}. By application of the above convergence theorem we have $\lambda\in\Gamma^*(f)$. Moreover, $\|\lambda\|\le\lim_n\|\lambda_n\|=\gamma^*(f)$, and so indeed $\|\lambda\|=\gamma^*(f)$.

For the case of $\gamma_*(f)$ let $a>\gamma_*(f)$. Then $(a\cal E_1^+)\cap\Gamma_*(f)\neq\varnothing$. The consistency of $G$ means that the mapping $\lambda\mapsto\int G\nu\,d\lambda=\int G\lambda\,d\nu$ of $a\cal E_1^+$ into $[0,+\infty]$ is finite valued and vaguely continuous for every $\nu\in\cal E^+$. Consider any vaguely convergent net $(\lambda_\alpha)$ on  $a\cal E_1^+\cap\Gamma_*(f)$ and denote by $\lambda_0$ its vague limit. Then $\|\lambda_0\|\le\liminf_\alpha\|\lambda_\alpha\|\le a$. Using the second expression for $\Gamma_*(f)$ in its definition we have $\int G\lambda_\alpha\,d\nu\ge\int_*f\,d\nu$ for any $\nu\in\cal E^+$, and hence by the stated continuity
$$
\int G\lambda_0\,d\nu=\lim_\alpha\,\int G\lambda_\alpha\,d\nu\ge\int_*f\,d\nu.
$$
Thus $\lambda_0\in\Gamma_*(f)$ and indeed $\lambda_0\in a\cal E_1^+\cap\Gamma_*(f)$
because $\|\lambda_0\|\le a$. It follows that $a\cal E_1^+\cap\Gamma_*(f)$ is vaguely closed, and in fact vaguely compact along with $a\cal E^+_1$, cf.\ Remark \ref{remark2.0}. The vaguely l.s.c.\ function $\lambda\mapsto\|\lambda\|$ therefore attains its infimum $\le\|\lambda_0\|$ when considered on $a\cal E_1^+\cap\Gamma_*(f)$. That infimum is also the infimum of $\|\lambda\|$ considered on all of $\Gamma_*(f)$ because $\|\lambda\|>a\ge\|\lambda_0\|$ on $\Gamma_*(f)\setminus a\cal E^+_1$.
\end{proof}

\begin{remark}\label{remark3.9} Suppose that $G$ is consistent. For any $f\in\cal F^+$ with $\gamma^*(f)<+\infty$, resp.\ $\gamma_*(f)<+\infty$, we denote by
\begin{align*}
\Lambda^*(f)&=\bigl\{\lambda\in\Gamma^*(f):\ \|\lambda\|=\gamma^*(f)\bigr\},\\
\Lambda_*(f)&=\bigl\{\lambda\in\Gamma_*(f):\ \|\lambda\|=\gamma_*(f)\bigr\},
\end{align*}
the nonvoid equivalence class of measures of minimal energy in the convex subset $\Gamma^*(f)$, resp.\ $\Gamma_*(f)$, of $\cal E^+$, cf.\ \cite[Lemma 4.1.1]{Fu1}. We may write $\Lambda(f)$ in place of $\Lambda^*(f)$ or $\Lambda_*(f)$ if these are equal.
\end{remark}
\begin{theorem}\label{thm3.10} Suppose that $G$ is consistent. The upper dual energy capacity $\gamma^*$ is sequentially order continuous from below, that is,
$$
\gamma^*(\sup_n\, f_n)=\sup_n\,\gamma^*(f_n)
$$
for any increasing sequence of functions $f_n\in\cal F^+$. The lower dual energy capacity $\gamma_*$ is sequentially order continuous from below on $\cal E^+$-measurable functions, that is:
$$
\gamma_*(\sup_n\, f_n)=\sup_n\,\gamma_*(f_n)
$$
for any increasing sequence of $\cal E^+$-measurable functions $f_n\in\cal F^+$.
\end{theorem}

\begin{proof} Consider first the case of $\gamma^*$ and write $f=\sup_n\,f_n$. In proving the non-trivial inequality $\gamma^*(f)\le\sup_n\,\gamma^*(f_n)$ we may suppose that the sequence $(\gamma^*(f_n))$ is bounded. By Theorem \ref{thm3.7} there exists $\lambda_n\in\Gamma^*(f_n)$ such that $\|\lambda_n\|=\gamma^*(f_n)$. By application of the convergence theorem (Theorem \ref{thm3.6}) we obtain a vague cluster point $\lambda\in\cal E^+$ of the sequence $(\lambda_n)$ such that
$$
G\lambda\ge\liminf_n\,G\lambda_n\ge\liminf_nf_n=f\text{\ q.e.},
$$
that is, $\lambda\in\Gamma^*(f)$. It follows that
$$
\gamma^*(f)\le\|\lambda\|\le\limsup_n\,\|\lambda_n\|
=\limsup_n\,\gamma^*(f_n)=\sup_n\,\gamma^*(f_n).
$$
The proof is similar in the case of $\gamma_*$, now with quasi-everywhere replaced by nearly everywhere while invoking the measurability requirement, cf. the last representation of $\Gamma^*(f)$ in the beginning of the present section.
\end{proof}

\subsection{Positive definite kernel.}

\begin{definition}\label{def3.4} A (symmetric) kernel $G:X\times X\to[0,+\infty]$ is \textit{positive definite\/} if the following three equivalent conditions are fulfilled for every $\mu,\nu\in\cal E^+$:
\begin{itemize}
\item[\rm(i)] $\int G\mu\,d\nu\le\|\mu\|\|\nu\|$,
\item[\rm(ii)] $2\int G\mu\,d\nu\le\int G\mu\,d\mu+\int G\nu\,d\nu$,
\item[\rm(iii)] $c(G\mu)\le\|\mu\|$ (hence actually $c(G\mu)=\|\mu\|$).
\end{itemize}
\end{definition}
Note that (i) obviously implies (ii). The converse (the Cauchy-Schwarz inequality) is of course obtained by replacing $\mu$ with $t\mu$ and $\nu$ with $t^{-1}\nu$, $0<t<+\infty$, and next minimizing over $t$. In (iii), $G\mu\in\cal G$ is $c$-capacitable, and the opposite inequality in (iii) holds for every symmetric kernel $G$ because $c(G\mu)\ge\int G\mu\,d(\mu/\|\mu\|)=\|\mu\|$ if $\|\mu\|>0$, noting that then $\mu/\|\mu\|\in\cal E^+_1$.  From (i) applied with $\nu\in\cal E^+_1$ follows by (\ref{2.2}) the stated inequality in (iii):  $c(G\mu)=c_*(G\mu)\le\|\mu\|$. Conversely, from the stated inequality in (iii) it follows that $\int G\mu\,d(\nu/\|\nu\|)\le c(G\mu)=\|\mu\|$, which implies (i) since we may suppose in (i) that $\nu\ne 0$ and hence $\|\nu\|\ne0$, cf.\ Remark \ref{remark2.0}.

In either condition (i) or (ii) it suffices to consider measures $\mu,\nu\in\cal E^+$ whose supports are compact and disjoint, \cite[Lemme 1]{Ch3}.

If $G$ is positive definite we denote by $\cal E$ the prehilbert space of all signed (Radon) measures $\mu$ on $X$ for which $\int G\mu\,d\mu<+\infty$, taking for inner product of measures $\mu,\nu\in\cal E$ the mutual energy $\int G\mu\,d\nu$ with $G\mu$ defined $|\mu|$-a.e.\ by the same formula as in case $\mu\ge0$, whereby  $G\mu$ becomes $\mu$-integrable. We keep the notation $\|\mu\|$ for the corresponding (energy) seminorm on $\cal E$.

A positive definite kernel $G$ is said to be \textit{strictly positive definite} (or to satisfy the \textit{energy principle}) if, for any $\mu\in\cal E$, $\|\mu\|=0$ implies $\mu=0$, or equivalently: if $G\mu=0$ $\mu$-a.e.\ implies $\mu=0$ (but here it may no longer be sufficient to consider measures $\mu\in\cal E$ of compact support). Thus a positive definite kernel $G$ is strictly positive definite if and only if the energy seminorm $\|\cdot\|$ is a norm. In the affirmative case the energy norm topology on the prehilbert space $\cal E$ (and its induced topology on $\cal E^+$) is also called the {\textit{strong topology}}.

Our standard hypothesis that $G$ be strictly positive on the diagonal is automatically fulfilled if $G$ is strictly positive definite. In fact, if $G(x,x)=0$ for some $x\in X$ then $\eps_x$ has finite energy $\int G\eps_x\,d\eps_x=G(x,x)=0$ and so $\eps_x=0$, which is impossible.

As mentioned in Section 1, consistency of a strictly positive definite kernel amounts to $\cal E^+$ being {\it strongly complete} and such that the strong topology on $\cal E^+$ is finer than the induced vague topology, \cite[Section 3.3]{Fu1}.

\begin{lemma}\label{lemma3.5} Each of the following inequalities holds if and only if $G$ is positive definite:
\begin{itemize}
\item[\rm(i)] $c(f)\le\gamma(f)$ for every $f\in\cal H_0$ (hence actually $c(f)=\gamma(f)$ even for $f\in\cal H_0^*$ in view of Lemma\/~{\rm\ref{lemma3.2}}),
\item[\rm(ii)] $c^*(f)\le \gamma^*(f)$ for every $f\in\cal F^+$,
\item[\rm(iii)] $c_*(f)\le \gamma_*(f)$ for every $f\in\cal F^+$.
\end{itemize}
\end{lemma}

\begin{proof} Recall from Lemma \ref{lemma3.2} that $c(f)\ge\gamma(f)$ for any $f\in\cal H^*_0$. If $G$ is positive definite we have for any $\lambda\in\Gamma_*(f)$ the 
equality $c_*(f-G\lambda)^+=0$ and hence the inequalities $c_*(f)\le c(G\lambda)\le \|\lambda\|$, invoking also \cite[Theorem 7.4]{Fu4} and Definition \ref{def3.4} (iii). This establishes (iii) in the present lemma, and (ii) is obtained similarly, invoking now \cite[Lemma 1.3 (b)]{Fu4}. Conversely, (ii) or (iii) implies (i) in view of Lemma \ref{lemma3.2}. Finally, (i) implies that $G$ is positive definite. In fact, for any $\lambda\in\cal E^+$, we have by (\ref{2.2})
$$
c(G\lambda)=\sup\,\bigl\{c(h):\ h\in\cal H_0,\ h\le G\lambda\bigr\}\le \|\lambda\|
$$
because $c(h)\le\gamma(h)\le \|\lambda\|$ for every competing $h$.
\end{proof}

We now justify the terms lower and upper dual capacity (cf.\ (\ref{2.2}) and (\ref{2.3})):

\begin{theorem}\label{thm3.12} Suppose that $G$ is consistent and positive definite. For any $f\in\cal F^+$ we have
\begin{align}
\gamma_*(f)&=\sup\,\bigl\{\gamma(h):\ h\in\cal H_0,\ h\le f\bigr\}\le c_*(f),\label{3.1a}\\
\gamma^*(f)&=\inf\,\bigl\{\gamma(g):\ g\in\cal G,\ g\ge f\bigr\}\le c^*(f),
\label{3.1c}\\
\gamma^*(f)&=\min\,\bigl\{\gamma(g):\ g
\in\cal G,\ g\ge f\text{\ q.e.}\bigr\}.
\label{3.1b}
\end{align}
\end{theorem}

\begin{proof} Consider first the case of $\gamma_*(f)$. Recall from Remark \ref{remark3.8a} that $\gamma^*(g)=\gamma_*(g)$ for any $g\in\cal G$ (or just $g\in\cal G^*$).
Since $\cal H_0$ is stable under supremum of finite families, in particular upper directed, so is $H:=\{h\in\cal H_0:\ h\le f\}$. For any $h\in H$ there exists by Theorem \ref{thm3.7} $\lambda_h\in\Lambda_*(h)$. We may assume that the numbers $\gamma(h)=\|\lambda_h\|$ remain bounded, say $\le a$, for $h\in H$. By Remark \ref{remark2.0} the net $(\lambda_h)_{h\in H}$ has a vague cluster point $\lambda\in a\cal E^+_1$. As at the end of the proof of Theorem \ref{thm3.6} we invoke the vague continuity of the function $\lambda\mapsto\int G\nu\,d\lambda$ on $a\cal E^+$ for any $\nu\in\cal E^+$. It follows that
$$
\int G\lambda\,d\nu\ge\liminf_{h\in H}\,\int G\lambda_h\,d\nu\ge\lim_{h\in H}\,\int h\,d\nu=\int_*f\,d\nu,
$$
where the latter inequality follows from $\lambda_h\in\Lambda_*(h)\subset\Gamma_*(h)$, hence $\int G\lambda_h\,d\nu\ge\int h\,d\nu$ (also note that the net $(\int h\,d\nu)_{h\in H}$ is increasing). This shows that $\lambda\in\Gamma_*(f)$. By the vague lower semicontinuity of energy
$$
\|\lambda\|\le\limsup_{h\in H}\,\|\lambda_h\|=\lim_{h\in H}\,\gamma(h)=\sup_{h\in H}\,\gamma(h)<+\infty
$$
and hence $\gamma_*(f)\le\|\lambda\|\le\sup_{h\in H}\,\gamma(h)\le\gamma_*(f)$ ($\le\gamma^*(f)$). This leads to (\ref{3.1a}):
$$
\gamma_*(f)=\sup_{h\in H}\,\gamma(h)\le\sup_{h\in H}\,c(h)=c_*(f)
$$
because $\gamma(h)\le c(h)$ for $h\in\cal H^*_0$ by Lemma \ref{lemma3.2}.

Next, consider (\ref{3.1b}).
We may assume that $\gamma^*(f)<+\infty$. By Theorem \ref{thm3.7}
$$
\gamma^*(f)=\min\,\bigl\{\|\lambda\|:\ \lambda\in\Gamma^*(f)\bigr\}
=\min\,\bigl\{\gamma(G\lambda):\ G\lambda\ge f\text{\ q.e.}\bigr\}
$$
in view of Lemma \ref{lemma3.3} and the definition of $\Gamma^*(f)$. Thus there exists $\lambda\in\cal E^+$ with $G\lambda\ge f$ q.e.\ and $\gamma(G\lambda)=\gamma^*(f)$. For any $g\in\cal G$ with $g\ge f$ q.e.\ we have $\gamma(g)\ge\gamma^*(f)$ because $\gamma^*(f)$ only depends on the $c^*$-equivalence class of $f$. Since $G\lambda\in\cal G$ it follows that
$$
\gamma^*(f)=\gamma(G\lambda)=\inf\,\bigl\{\gamma(g):\ g\in\cal G,\ g\ge f\text{\ q.e.}\bigr\},
$$
and this infimum is attained by $g=G\lambda$, which proves (\ref{3.1b}).

Having fixed $\lambda\in\Lambda^*(f)$, we have $c^*((f-G\lambda)^+)=0$, and therefore for any $\eps>0$ there exists a function $g_\eps\in\cal G$ with $c(g_\eps)<\eps$ such that $(f-G\lambda)^+\le g_\eps$. Writing $g=G\lambda+g_\eps\in\cal G$ we infer that $f\le g$. When $G$ is positive definite the seminorm $\|\cdot\|$ is subadditive, and it easily follows that so is $\gamma^*$. Hence
$$
\gamma(g)\le\gamma(G\lambda)+\gamma(g_\eps)\le\|\lambda\|+c(g_\eps)<\gamma^*(f)+\eps
$$
in view of Lemmas \ref{lemma3.2} and\ref{lemma3.3} because
$\gamma(g_\eps)\le c(g_\eps)$ by
(\ref{3.1a}) and because $\|\lambda\|=\gamma^*(f)$ since $\lambda\in\Lambda^*(f)$. By varying $\eps$ this implies the equality in (\ref{3.1c}). And that equality implies the inequality in (\ref{3.1c}) by inserting $\gamma(g)\le c(g)$ and invoking the definition of upper energy capacity $c^*(f)$.
\end{proof}

\begin{cor}\label{cor3.13} Suppose that $G$ is consistent and positive definite. Then $c^*(f)=\gamma^*(f)$ and  $c_*(f)=\gamma_*(f)$ for any $f\in\cal F^+$. Hence $c^*,c_*:\cal F^+\to[0,+\infty]$ are sequentially order continuous from below (on $\cal E^+$-measurable functions in the case of $c_*$):
$$
c^*(\sup_n\,f_n)=\sup_n\,c^*(f_n),\quad c_*(\sup_n\,f_n)=\sup_n\,c_*(f_n)
$$
for any increasing sequence of functions $f_n\in\cal F^+$ ($\cal E^+$-measurable functions $f_n\in\cal F^+$ in the case of $c_*$).
\end{cor}

The identity $c_*=\gamma_*$ follows by combining the inequality $\gamma_*\le c_*$ in (\ref{3.1a}) (for $G$ consistent) with the inequality $c_*\le\gamma_*$ obtained in Lemma \ref{lemma3.5} (iii) (for $G$ positive definite). The identity $c^*=\gamma^*$ now follows from (\ref{3.1c}). Hence the remaining assertions follow from the corresponding sequential order continuity of $\gamma^*$ and $\gamma_*$ (Theorem \ref{thm3.10}).

For the concepts of \textit{quasi u.s.c.}, resp.\ \textit{quasi l.s.c.}, functions of class $\cal F^+$ we refer to \cite[Section 3]{Fu3}, or see Section \ref{quasi} below. It follows from the definitions that a set $A\subset X$ is quasiclosed, resp.\ quasiopen, if and only if the indicator function $1_A$ is quasi u.s.c., resp.\ quasi l.s.c.

\begin{lemma}\label{lemma3.8} Suppose that $G$ is consistent and positive definite.
\begin{itemize}
\item[\rm(a)] Every function $f\in\cal F^+$ such that $c^*(f)<+\infty$ has a majorant of class $\cal H_0^*$. If moreover $f$ is quasi u.s.c.\ then $f\in\cal H^*_0$.
\item[\rm(b)] Every set $A\subset X$ such that $c^*(A)<+\infty$ is contained in some quasicompact set. If moreover $A$ is quasiclosed then $A$ is quasicompact.
\end{itemize}
\end{lemma}

\begin{proof} (a) Since $\gamma^*(f)=c^*(f)<+\infty$ there exists $\lambda\in\Gamma^*(f)$. Thus $G\lambda\ge f$ q.e., and $G\lambda\in\cal H^*_0$ by Definition \ref{def3.5}. By redefining $G\lambda$ on the set $\{G\lambda<f\}$  we obtain the desired majorant of $f$ of class $\cal H^*_0$. If $f$ itself is quasi u.s.c.\ then $f\in\cal H^*_0$
by \cite[Theorem 2.5 (b)]{Fu4}.

(b) Since $c^*(1_A)<+\infty$, $1_A$ has by (a) a majorant $h\in\cal H^*_0$ . Then $h$ is quasi u.s.c., again by \cite[Theorem 2.5 (b)]{Fu4}. It follows by \cite[Lemma 3.3]{Fu3} that the set $H:=\{h\ge1\}$ is quasiclosed (and contains $A$). Hence
the function $1_H$ is quasi u.s.c.\ (and $\le h$). Since $h\in\cal H^*_0$ it follows by the latter assertion (a) that $1_H\in\cal H^*_0$, which by \cite[Lemma 2.4]{Fu4} means that the set $H$ (which contains $A$) is quasicompact. If $A$ itself is quasiclosed then $1_A$ is quasi u.s.c.\ and majorized by $h\in\cal H^*_0$, and it follows by the second assertion (a) that $1_A\in\cal H^*_0$, and again by \cite[Lemma 2.4]{Fu4} that $A$ is quasicompact.
\end{proof}

\subsection{Further properties of measures $\mu\in{\rm{M}}(f)$.} Let us now return to Theorem \ref{thm2.1} and establish further properties of the capacitary measures $\mu\in{\rm{M}}(f)$, now with $G$ consistent and positive definite. We shall repeatedly use the above Lemma \ref{lemma3.8}.

\begin{theorem}\label{thm3.9} Suppose that $G$ is consistent and positive definite. Let $f\in\cal H_0^*$, or equivalently: $f$ is quasi u.s.c.\ with $c^*(f)<+\infty$. The class ${\rm{M}}(f)$ of all capacitary measures for $f$ is a nonvoid vaguely compact convex subset of $\cal E^+$ and consists of all measures of class $\cal E^+$ satisfying {\rm{(a)}} and {\rm{(b)}} in Theorem {\rm\ref{thm2.1}}. These measures (which are carried by $\{f>0\}$) have, in addition to {\rm{(a)}}, {\rm{(b)}}, {\rm{(c)}} in Theorem {\rm\ref{thm2.1}}, the following properties:
\begin{itemize}
\item[\rm(d)]
  ${\rm{M}}(f)\subset\Lambda(f)$, and $\Lambda(f)$ is an equivalence class in $\cal E^+$. In particular, if $G$ is strictly positive definite then ${\rm{M}}(f)$ and $\Lambda(f)$ reduce to the same single capacitary measure.
\end{itemize}
\end{theorem}

\begin{proof} For any measure $\mu\in\cal E^+$ satisfying (a) and (b) we have
$$
c(f)^2\le c(G\mu)^2=\int G\mu\,d\mu=\int f\,d\mu,
$$
the inequality by (a), the former equality by Definition \ref{def3.4} (iii), and the latter equality by (b). On the other hand, $c(f)^2\ge\int f\,d\mu$ by (\ref{2.5}) because $\mu\in\cal E^+$ and $G\mu\le f$ $\mu$-a.e., by (b). Thus $\mu$ is maximizing in (\ref{2.5}), and hence indeed $\mu\in{\rm{M}}(f)$ according to Theorem \ref{thm2.1}.

Concerning (d), for the equality $c(f)=\gamma(f)$ see Lemma \ref{lemma3.5} (i). It follows from (a) that $\mathrm M(f)\subset\Gamma^*(f)$. Since, by (c),
$\|\mu\|=c(f)=\gamma(f)$ for any $\mu\in\mathrm M(f)$ we thus have
$\mathrm M(f)\subset\Lambda(f)$,
cf.\ Remark \ref{remark3.9}, where it is noted that $\Lambda(f)$ is an equivalence class in $\cal E^+$. Convexity of ${\rm{M}}(f)$ amounts to that of ${\rm{M}}_1(f)$, cf.\ Theorem \ref{thm2.1}, which reduces to that of $\cal E^+_1$ (a consequence of positive definiteness of $G$). This establishes (d).
\end{proof}

\subsection{Capacitability of Suslin functions.}
\begin{theorem}\label{thm3.14} Suppose that $G$ is consistent and positive definite. Then the following two assertions hold.
\begin{itemize}
\item[\rm(a)] For any decreasing net of u.s.c.\ functions $h_\alpha\in\cal F^+$ with $c^*(h_\alpha)<+\infty$ we have
$$
c^*(\inf_\alpha\,h_\alpha)=\inf_\alpha\,c^*(h_\alpha).
$$
\item[\rm(b)] For any decreasing sequence of quasi u.s.c.\ functions $h_n\in\cal F^+$ with $c^*(h_n)<+\infty$ we have
$$
c^*(\inf_n\,h_n)=\inf_n\,c^*(h_n).
$$
\end{itemize}
\end{theorem}

\begin{proof} Assertion (a) follows from \cite[Theorem 3.6]{Fu4} (a) because $c^*(h_\alpha)<+\infty$ and hence each $h_\alpha$ has a majorant of class $\cal H^*_0$ according to the former assertion in Lemma \ref{lemma3.8} (a), so that \cite[Theorem 2.5 (b)]{Fu4} applies.
For (b) of the present theorem, the latter assertion of Lemma \ref{lemma3.8} (a) shows that $h_n\in\cal H^*_0$ and hence \cite[Theorem 3.6]{Fu4} (c) applies.
\end{proof}

\begin{theorem}\label{thm3.15} Suppose that $G$ is consistent and positive definite. Every $\cal H_0$-Suslin function $f\in\cal F^+$ is $c$-capacitable.
\end{theorem}

\begin{proof} This follows from the sequential order continuity of $c^*$ from below on functions of class $\cal F^+$ (Corollary \ref{cor3.13}) together with the sequential order continuity of $c$ from above on functions of class $\cal H_0$
(Theorem \ref{thm3.14} (b)). We merely have to apply Choquet's capacitability theorem in its abstract form \cite{Ch4}, cf.\ \cite[Section 1.6 and Remark 2 in Section 4.5]{Fu4}.
\end{proof}

\subsection{Upper capacitary measures.} Throughout this subsection it is assumed that $G$ is consistent and positive definite, and also that $X$ {\textit{has a countable base of open sets}} (the second axiom of countability). Passing from the given function $f\in\cal H^*_0$ in Theorems \ref{thm2.1} and \ref{thm3.9} to an arbitrary function $f\in\cal F^+$ of finite upper capacity $c^*(f)$, we introduce the {\textit{upper capacitary measures}} for $f$.

For any $f\in\cal F^+$ define
$$
\Phi^*(f)=\bigl\{\phi\in\cal F^+:\ \phi\text{\ is quasi u.s.c.,\ }\phi\ge f\text{\ q.e.}\bigr\}.
$$
Clearly, $\Phi^*(f)$ is convex and stable under countable infimum. Furthermore, $\Phi^*(f)$ is a union of $c^*$-equivalence classes in $\cal F^+$, and only depends on the $c^*$-equivalence class  of $f$ in $\cal F^+$. By Corollary \ref{cor3.13}, $c^*$ is sequentially order continuous from below. Since $X$ has a countable base of open sets it therefore follows from \cite[Theorem 3.4]{Fu3} that $f$ has a {\textit{quasi u.s.c.}\ (upper) {\textit{envelope} $f^*$, that is,
a quasiminimal element of $\Phi^*(f)$. Explicitly, $f^*$ is an element of $\Phi^*(f)$ which is majorized q.e.\ by any element of $\Phi^*(f)$. If $c^*(f)<+\infty$, that is $\gamma^*(f)<+\infty$, then every quasi u.s.c.\ envelope $f^*$ of $f$ is of class $\cal H^*_0$, being majorized q.e.\ by $G\lambda\in\cal H^*_0$ for any $\lambda\in\Lambda^*(f)\ne\varnothing$
(cf.\ Lemma \ref{lemma3.8} (a))}. Any two quasi u.s.c.\ envelopes of $f$ are $c^*$-equivalent.

\begin{definition}\label{def3.16} Let $X$ have a countable base of open sets and let $G$ be consistent and positive definite. Let $f\in\cal F^+$ and suppose that $c^*(f)<+\infty$.  By an {\textit{upper capacitary measure}} for $f$ we understand a measure $\mu\in\cal E^+$ whose potential has the following two properties, $f^*$ denoting a chosen upper envelope of $f$:
\begin{itemize}
\item[\rm(a)] $G\mu\ge f$ q.e., that is, $G\mu\in\Phi^*(f)$,
\item[\rm(b)] $G\mu\le f^*$ $\mu$-a.e.
\end{itemize}
We denote by ${\rm{M}}^*(f)$ the set of all upper capacitary measures for $f$.
\end{definition}

\begin{theorem}\label{thm3.17} Suppose that $X$ has a countable base and that $G$ is consistent and positive definite. Let $f\in\cal F^+$, and suppose that $c^*(f)<+\infty$. For any quasi u.s.c.\ envelope $f^*$ of $f$ we have
\begin{eqnarray*}
\quad c^*(f)=c(f^*),\quad\Lambda^*(f)=\Lambda(f^*),\quad{\rm{M}}^*(f)={\rm{M}}(f^*),\label{3.2}
\end{eqnarray*} where ${\rm{M}}(f^*)$ denotes the nonvoid vaguely compact convex set of all capacitary measures for $f^*\in\cal H^*_0$, cf.\ Theorems {\rm\ref{thm2.1}} and {\rm\ref{thm3.9}} (now with $f$ replaced by $f^*$).
\end{theorem}

\begin{proof} With $\lambda\in\Lambda^*(f)$ we have $G\lambda\in\Phi^*(f)$ and $\|\lambda\|=\gamma^*(f)=c^*(f)$ (Remark \ref{remark3.9} and Corollary \ref{cor3.13}).
Since $f^*$ is a quasiminimal function in $\Phi^*(f)$ and $G\lambda\in\Phi^*(f)$ it follows that $G\lambda\ge f^*$ q.e.
Hence $c^*(f)=\|\lambda\|=c(G\lambda)\ge c(f^*)\ge c^*(f)$, and so indeed $c^*(f)=c(f^*)$.

Next, let $\lambda\in\Lambda(f^*)$ and $\lambda^*\in\Lambda^*(f)\subset\Gamma^*(f)\subset\Gamma(f^*)$. Here, $\Gamma(f^*):=\Gamma^*(f^*)=\Gamma_*(f^*)$, the latter equality by Lemma \ref{lemma3.2}.
Thus $\lambda^*\in\Gamma(f^*)$, and $\lambda$ is even energy minimizing in $\Gamma(f^*)$. It therefore follows by \cite[Lemma 4.1.1]{Fu1} applied to $\Gamma=\Gamma(f^*)$ that
$$
\|\lambda^*-\lambda\|^2\le\|\lambda^*\|^2-\|\lambda\|^2=c^*(f)^2-c(f^*)^2=0,
$$
and hence any $\lambda^*$ in $\Lambda^*(f)$ is equivalent to some $\lambda$ in the equivalence class $\Lambda(f^*)$, cf.\ Theorem \ref{thm3.9} (d) (with $f$ replaced by $f^*$).
As $\Lambda^*(f)$ is an equivalence class as well (Remark \ref{remark3.9}), this yields $\Lambda^*(f)=\Lambda(f^*)$.

Finally, let $\mu\in{\rm{M}}^*(f)$. Then
 (a) in Definition \ref{def3.16} implies the stronger inequality $G\mu\ge f^*$ q.e.\ (as in the beginning of the proof),  in view of (\ref{2.3}). And this together with (b) in the same definition implies
 $\mu\in{\rm{M}}(f^*)$, so that ${\rm{M}}^*(f)\subset{\rm{M}}(f^*)$
 in view of the characterization of ${\rm{M}}(f^*)$ in Theorem \ref{thm3.9} (with $f$ replaced by $f^*$).
The converse inclusion is obvious, and so indeed ${\rm{M}}^*(f)={\rm{M}}(f^*)$.
\end{proof}

We omit the analogous consideration of the {\textit{lower}} capacitary measures for $f$.

\section{Balayage (sweeping) on a quasiclosed set}\label{sec4}

For sweeping of a (positive) measure $\omega$ (not necessarily of finite energy) onto a suitable set $A\subset X$ we shall apply Theorem \ref{thm3.9} to the function $f:=1_AG\omega$. For that purpose we need that $f\in\cal H^*_0$, in particular that $c^*(1_AG\omega)<+\infty$.

\subsection{Quasicontinuous functions.}\label{quasi}
A map $f$ of a subset $U$ of $X$ into a topological space $T$ is said to be {\textit{quasicontinuous}} (with respect to outer energy capacity $c^*$) if there exists for any $\eps>0$ a set $V\subset U$ with $c^*(V)<\eps$ such that the restriction of $f$ to $U\setminus V$ is continuous (in the extended sense if e.g.\ $T=[0,+\infty]$). One may then clearly take $V$ open. {\textit{Quasi u.s.c}}.\ functions of class $\cal F^+$ are defined similarly, replacing `continuous' by `u.s.c.' (twice). As noted earlier, a set $A\subset X$ is quasiclosed if and only if $1_A\in\cal F^+$ is quasi u.s.c.

\begin{remark}\label{remark4.0} If $G$ is consistent and $\omega\in\cal E^+$ then $G\omega\in\cal H^*_0$ (Definition \ref{def3.5}), and in particular $G\omega$ is quasi u.s.c.\ by \cite[Theorem 2.5 (b)]{Fu4} (and even quasicontinuous), and hence $1_AG\omega$ is quasi u.s.c.\ for any quasiclosed set $A\subset X$.
\end{remark}

\begin{lemma}\label{lemma4.0} Suppose that $G$ is consistent, and positive definite.
Let $\omega$ be a measure on $X$ such that $G\omega$ is quasicontinuous, and let $A\subset X$ be quasiclosed. Write $f:=1_AG\omega$.
\begin{itemize}
\item[\rm(a)] If $c^*(f)<+\infty$ then $f\in\cal H_0^*$.
\item[\rm(b)] If $X$ has a countable base then $f$ is $c$-capacitable: $c^*(f)=c_*(f)$ $({}=c(f))$.
\end{itemize}
\end{lemma}

\begin{proof} Clearly $f$ is quasi u.s.c.

(a) Since $c^*(f)<+\infty$ we have $\gamma^*(f)=c^*(f)<+\infty$ (Corollary \ref{cor3.13}) and hence there exists a measure $\lambda\in\cal E^+$ such that $G\lambda\ge f$ q.e.\ (and $\|\lambda\|=c^*(f)$). Then $G\lambda\in\cal H_0^*$ (Definition \ref{def3.5}). By redefining $h:=G\lambda$ on a set of zero outer $c$-capacity, whereby $h$ remains of class $\cal H_0^*$, we achieve that $h\ge f$ everywhere in $X$, and since $f$ is quasi u.s.c.\
we conclude from \cite[Theorem 2.5 (b)]{Fu4} that indeed $f\in\cal H_0^*$.

(b) Having a countable base, $X$ is the union of an increasing sequence of compact sets $K_j$, e.g.\ \cite[Section 2, No.\ 1, Corollaire]{Bo2}. Furthermore, $G\omega$ is the pointwise supremum of an increasing sequence of finite continuous functions $\phi_j\ge0$, supported by $K_j$, e.g.\ \cite[Section 2, Proposition 11]{Bo2}. Observe that $\phi_j$, being finite and continuous, is bounded from above. The functions $f_j:=1_{A\cap K_j}\phi_j$ form an increasing sequence of bounded, quasi u.s.c.\ functions with the supremum $f$. Since $A\cap K_j$ is quasicompact, $1_{A\cap K_j}\in\cal H^*_0$ and hence $f_j\in\cal H^*_0$ in view of \cite[Lemma 2.4 and Theorem 2.5 (b)]{Fu4}. As noted shortly before Theorem \ref{thm2.0}, the functions $f_j$ are $\cal E^+$-measurable and $c$-capacitable, and therefore we conclude by Corollary \ref{cor3.13} that indeed $c^*(f)=c_*(f)$.
\end{proof}

\begin{lemma}\label{lemma4.1} Let $(U_n)$ be a cover of $X$ by an increasing sequence of open sets $U_n\subset X$
and let for each $n$, $f\bigl|_{U_n}$ be quasicontinuous, $f\in\cal F^+$ being fixed.
Then $f$ is quasicontinuous on all of $X$.
\end{lemma}

\begin{proof} By the above definition of quasicontinuity there exists for every $n$ an open set $V_n\subset U_n$ such that $f\bigl|_{U_n\setminus V_n}$ is continuous (in the extended sense) and that $c(V_n)<\eps/2^n$. Set $V:=\bigcup_nV_n$ (open in $X$). Then $c(V)\le\sum_n\,c(V_n)<\eps$. Since the open sets $U_n$ cover $X$, while $f\bigl|_{U_n\setminus V}$ is continuous, so is $f\bigl|_{\complement V}$.
\end{proof}

\subsection{Relations between the energy capacity and the standard $G$-capacity.}\label{relations}
In view of Lemma \ref{lemma4.0} and Remark \ref{remark4.0} we propose to find conditions under which the potential $G\omega$ of any
measure $\omega\,\in\cal M^+$
 is quasicontinuous (as a map into $[0,+\infty]$).
For the  proof of the following Theorem \ref{thm4.3} we shall need besides the energy capacity $c$ the \textit{$G$-capacity} $G\text{-}\!\capa$. While our study of the energy capacity $c$, provided above, requires a kernel to be symmetric and ${}>0$ on the diagonal, the $G$-capacity is often studied without these requirements. In this subsection we shall therefore drop these requirements and, unless stated otherwise, allow $G$ to be any l.s.c.\ function $X\times X\to[0,+\infty]$ ($X$ locally compact). Its {\textit{adjoint}} kernel $\check G$ is defined by $\check G(x,y)=G(y,x)$ for $x,y\in X$. The {\textit{symmetric part}} of $G$ is $\frac12(G+\check G)$ and will be denoted by $\widetilde G$.

A kernel $G$ is said to satisfy the \textit{continuity principle} (Evans-Vasilesco), or to be \textit{regular} (Choquet \cite{Ch1}), if for any measure $\mu\in\cal M^+$ of compact support $\supp\mu$, the potential $G\mu$ is finite and continuous on all of $X$ provided  that it is finite and continuous relative to $\supp\mu$.

A kernel $G$ is said to satisfy the {\textit{dilated maximum principle}} if there exists a constant $k\ge1$ such that, for any measure $\mu\in\cal M^+$ of compact support and with potential $G\mu\le1$ $\mu$-a.e., we have $G\mu\le k$ everywhere on $X$. When $k$ is specified we speak of the {\textit{$k$-dilated maximum principle}}. For $k=1$ this is simply called the (Frostman) maximum principle.

 \begin{theorem}\label{thm4.4a} If $G$ is finite off the diagonal and continuous in the extended sense on $X\times X$, and if the restriction of $G$ to $K\times K$ satisfies the dilated maximum principle for every compact $K$, then $G$ is regular.
 \end{theorem}

This is due to Ohtsuka \cite{O1}, \cite{O2}, \cite[Eq.\ (1.10) (p.\ 155)]{O}, and independently Choquet \cite{Ch1} (without proof).

A  symmetric kernel $G$ which satisfies the maximum principle is positive definite, \cite{Ch3,N}.

The $G$-capacity of a compact set $K\subset X$ is defined by
\begin{equation}\label{G-capa}G\text{-}\!\capa(K)=\sup\,\bigl\{\nu(K):\ \nu\in\cal M^+,\ G\nu\le1\ \text{ everywhere}\bigr\},
\end{equation}
see \cite{Ch2}, \cite[p.\ 43]{Br1}, \cite{O}, \cite{Fu2a}. We further assume that $G$ is \textit{non-degenerate} (in the first variable) in the sense that $G(.,y)\not\equiv0$ for every $y\in X$. (Of course, this assumption is satisfied if $G$ is strictly positive on the diagonal.) Equivalently, $G\mu\not\equiv0$ for every non-zero $\mu\in\cal M^+$. In fact, for given $y\in\supp\mu$ choose $x\in X$ so that $G(x,y)>0$, and open neighborhoods $U$ of $x$ and $V$ of $y$ so that $G>0$ on $U\times V$, and hence $G\mu(x)>0$. Conversely, take $\mu=\eps_y$, hence $G(x,y)=G\mu(x)>0$. We show that this second definition further is equivalent to  the following third definition:
$$
\cal S:=\{\mu\in\cal M^+:G\mu\le1 \text{ on }X\}\quad\text{is vaguely compact.}
$$
In the first place, $\cal S$ is clearly vaguely closed. By the first definition, for any $y\in X$ there exists $x_y\in X$ with $G(x_y,y)>0$. Hence there is an open neighborhood $V_y$ of $y$ such that $G(x_y,z)\ge\tfrac12G(x_y,y)$ for all $z\in V_y$. It follows that $\tfrac12G(x_y,y)\mu(V_y)\le\int_{V_y}\,G(x_y,z)\,d\mu(z)\le G\mu(x_y)\le1$
for every $\mu\in\cal S$, and so $\mu(V_y)$ is bounded as a function of $\mu\in\cal S$. Since finitely many $V_y$ cover $K$ this shows that the function $\mu\mapsto\mu(K)$ indeed is bounded on $\cal S$, and so $\cal S$ is compact. Conversely, if $\cal S$ is compact it cannot contain a ray $\{t\mu:t\in[0,+\infty[\,\}$, as it would if $G\mu\equiv0$ for some non-zero $\mu\in\cal M^+$.

This proves that the supremum in the definition of $G\text{-}\!\capa(K)$ is attained and hence finite when $G$ is non-degenerate.

From now on, $G$ is assumed to be strictly positive on the diagonal; then so is $\widetilde G$.
Recall that the energy capacity of a compact set $K\subset X$ for a (symmetric) kernel $\widetilde G$, strictly positive on the diagonal, is characterized
by either of the following two equalities:
\begin{align}\label{4.0'}
c_{\widetilde G}(K)&=\max\,\Bigl\{\nu(K):\ \nu\in\cal E^+,\ \int \widetilde{G}\nu\,d\nu\le1\Bigr\}\ (<+\infty),\\
\label{4.0''}
c_{\widetilde G}(K)^2&=\max\,\bigl\{\nu(K):\ \nu\in\cal E^+,\ \widetilde G\nu\le1\ \text{$\nu$-a.e.\ on $X$}\bigr\}\ (<+\infty),
\end{align}
(cf.\ Theorem \ref{thm2.0} or Eq.\ (\ref{2.5}), respectively, for the kernel $\widetilde G$ and functions of class $\cal H_0$, in particular for (indicator functions for) compact sets). Then every maximizing measure $\nu$ for any of (\ref{G-capa}), (\ref{4.0'}), (\ref{4.0''})
is  clearly carried by $K$ and has finite energy. We show that
\begin{equation}\label{4.1'}
c_{\widetilde G}(K)^2\ge G\text{-}\!\capa(K).
\end{equation}
Let $\nu$ be a maximizing measure for $G\text{-}\!\capa(K)$. Then $\int G\nu\,d\nu\le\nu(K)$. We may assume that $\int G\nu\,d\nu>0$, for otherwise $\nu=0$ by Remark \ref{remark2.0}, and hence $G\text{-}\!\capa(K)=0$. Clearly, $\int G\nu\,d\nu=\int\widetilde G\nu\,d\nu$, also denoted by $\|\nu\|^2$. Writing $\mu:=\nu/\|\nu\|$ we have $\|\mu\|=1$, so $\mu$ competes for $c_{\widetilde G}(K)$ in (\ref{4.0'}). It follows that
$$
c_{\widetilde G}(K)\ge\mu(K)=\frac{\nu(K)}{\|\nu\|}\ge\sqrt{\nu(K)}=\sqrt{G\text{-}\!\capa(K)}
$$
because $\|\nu\|^2=\int G\nu\,d\nu\le\nu(K)$. This establishes (\ref{4.1'}).

If $G$ satisfies the $k$-dilated maximum principle we obtain an inequality in the opposite direction. Let $\nu$ be maximizing for $c_{\widetilde G}(K)^2$ in (\ref{4.0''}).
Since $G\nu\le 2\widetilde G\nu$ it follows that $G\mu\le1$ everywhere, writing $\mu:=\nu/(2k)$. Thus $\mu$ competes for $G\text{-}\!\capa(K)$, and so $G\text{-}\!\capa(K)\ge\mu(K)=\nu(K)/(2k)=c_{\widetilde G}(K)^2/(2k)$, whence
\begin{equation}\label{4.0}
  c_{\widetilde G}(K)^2\le2k\,G\text{-}\!\capa(K).
\end{equation}

In \cite{Ch1}, \cite{Ch2}, Choquet has studied relations between $G\text{-}\!\capa$ and $\check G\text{-}\!\capa$ (without assuming $G(x,x>0$). It is interesting (and perhaps more or less new) that such relations have close counterparts in relations between $G\text{-}\!\capa$ and the energy capacity $c_{\widetilde G}$. The inequality (\ref{4.0}) is a counterpart to the inequality $\check G\text{-}\!\capa(K)\le 4k\,G\text{-}\!\capa(K)$ established in \cite[Lemme 4]{Ch2}. We proceed to establish two more advanced relationships (Theorems \ref{thm4.3} and \ref{thm4.3a} below) between $G\text{-}\!\capa$ and $c_{\widetilde G}$ (counterparts to \cite[Proposition 2, Th\'eor\`eme 1]{Ch2}).

Following \cite{Ch1} a point $x_0\in X$ is called a point of $k$-\textit{undulation} for $G$ ($k>1$ real)
if for any neighborhood $V$ of $x_0$ there exists a measure $\mu$ on a compact set $K\subset V$ such that $G\mu$ is bounded on $K$ and
$$
\sup_{x\in V}\,G\mu(x)\ge k\sup_{x\in K}\,G\mu(x).
$$
We denote by $\cal O_k$ the set of all points of $k$-undulation for $G$, and by $\cal O_\infty:=\bigcap_{k>1}\cal O_k$ the set of all points of so-called \textit{strong undulation}. Clearly, $\cal O_k$ and $\cal O_\infty$ are closed subsets of $X$, and every point of $\complement\cal O_\infty$ has an open neighborhood $V$ such that $G\bigl|_{V\times V}$ satisfies the dilated maximum principle.

\begin{definition}\label{stronglyregular} A regular kernel $G$ on $X$ is said to be \textit{strongly regular} if $G$ is finite and continuous off the diagonal and strictly positive on the diagonal, and if at least one limit point of $\cal O_\infty$ has a countable base of neighborhoods (e.g.\ if $X$ is first countable).
\end{definition}

The following theorem \ref{thm-und} was established by Ohtsuka \cite[Corollary 2, p.\ 170]{O}, inspired by a similar result announced by Choquet \cite{Ch1}, \cite{Ch2}.

\begin{theorem}\label{thm-und} If $G$ is strongly regular then $\cal O_\infty$ is discrete.
\end{theorem}

\begin{theorem}\label{thm4.3}
Let $X$ be compact. Suppose that $G$ is strongly regular and that $G(x,x)=+\infty$ for every point $x\in\cal O_\infty$. Then the outer $G$-capacity $G\text{-}\!\capa^*$ and the outer energy capacity $c^*_{\widetilde G}$ are simultaneously small. Explicitly, for any $\eps>0$ there exists $\eta>0$ such that for every subset $A$ of $X$
$$
G\text{-}\!\capa^*(A)<\eta\quad\text{implies}\quad c_{\widetilde G}^*(A)<\eps,
$$
and similarly with the two capacities interchanged.
\end{theorem}

\begin{proof} The latter case follows from (\ref{4.1'}) extended to outer capacities of arbitrary sets $A$. For the former case, according to Theorem \ref{thm-und} above, $\cal O_\infty$ is discrete, hence finite, by compactness of $X$. For any $x\in\cal O_\infty$, $G(x,x)=+\infty$ by hypothesis, that is, $\eps_x$ has infinite energy, hence $c_{\widetilde G}(\{x\})=0$, and altogether $c_{\widetilde G}(\cal O_\infty)=0$. For given $\eps>0$ fix an open set $\omega\supset\cal O_\infty$ so that $c_{\widetilde G}(\omega)<\eps/2$. The set $K:=\complement\omega$ is compact along with $X$.  The restriction of $G$ to $K\times K$ satisfies the $k$-dilated maximum principle for some constant $k\,=k(\eps)\,\ge1$, (cf.\ text just before Definition \ref{stronglyregular}). Now define $\eta={\eps^2}/{8k}$. For any $A\subset X$ with $G\text{-}\!\capa^*(A)<\eta$ we therefore have by extension of (\ref{4.0}) to outer capacity of arbitrary sets $c_{\widetilde G}^*(A\cap K)^2\le 2k\,G\text{-}\!\capa^*(A\cap K)\le\eps^2/4$, and so indeed
$$
c_{\widetilde G}^*(A)\le c_{\widetilde G}^*(A\cap\omega)+c_{\widetilde G}^*(A\cap K)\le\frac{\eps}2+\frac{\eps}2=\eps.
$$
\end{proof}

In Section \ref{quasi} we have defined `quasitopological' concepts such as a quasicontinuous function (with respect to outer energy capacity $c^*$
for a symmetric kernel). There are similar concepts with $c^*$ replaced by any other \textit{outer capacity} $C$ on $X$, cf.\ \cite[Section 1.5]{Fu3}.

\begin{theorem}\label{thm4.3a} Let $X$ be countable at infinity. Suppose that $G$ is strongly regular and that $G(x,x)=+\infty$ for every point $x\in\cal O_\infty$. The $G$-potential of any $\omega\in\cal M^+$ is quasicontinuous with respect to the (outer) energy capacity $c^*_{\widetilde G}(\cdot)$.
\end{theorem}

\begin{proof} The requirement that $X$ be countable at infinity means that $X$ can be covered by a sequence of compact subsets, and hence also by an increasing sequence of open relatively compact sets $U_n\subset X$ such that $\overline U_n\subset U_{n+1}$, e.g.\ \cite[Chap.\ 1, Sect.\ 9, Prop.\ 15]{Bo}. Write $G_n=G\bigl|_{\overline{U_n}\times\overline{U_n}}$ and denote by $\omega_n$ the restriction of $\omega$ to $\overline U_n$. Then $G_n$ is strongly regular along with $G$.  By \cite[Th\'eor\`eme 1]{Ch2} applied to the regular kernel $G_n$, the potential $G_n\omega_n$ is finite q.e.\ and quasicontinuous on $\overline U_n$ with respect to the $G$-capacity relative to the kernel $\check G_n$ on the compact space $\overline U_n$, in view of Theorem \ref{thm4.3} applied to the kernel $\check G_n$ in place of $G$, and finally also with respect to the energy capacity $c^*_{\widetilde G}$ on $X$. In fact (unlike what would be the case for $\check G\text{-}\!\capa$), we have for any compact set $K\subset\overline U_n$,  $c_{\widetilde{G}}(K)=c_{\widetilde{G}_n}(K)$ because the energy $\int G\nu\,d\nu$ of any measure $\nu$ on $K$, and hence the corresponding Gauss integral, only depend on the restriction of $G$ to $K\times K$.
 For any open set $V\subset U_n$  we therefore have $c_{\widetilde{G}}(V)=c_{\widetilde{G}_n}(V)$ (open sets being obviously capacitable with respect to energy capacity). Thus $G\omega_n=G_n\omega_n$ is quasicontinuous on $U_n$ with respect to $c^*_{\widetilde G}(\cdot)$. And $G(\omega-\omega_n)$ is even continuous on $U_n$ (which does not meet $\supp(\omega-\omega_n)$) in view of the finiteness and continuity of $G$ off the diagonal. Consequently, the sum $G\omega$ of these two quasicontinuous $G$-potentials on $U_n$ is quasicontinuous on $U_n$ with respect to $c^*_{\widetilde G}$. Finally, by Lemma \ref{lemma4.1}, $G\omega$ is indeed quasicontinuous on all of $X$ with respect to $c^*_{\widetilde G}$.\end{proof}

\begin{cor}\label{cor4.3b} Under the hypotheses of Theorem\/ {\rm\ref{thm4.3a}} we have for any subset $E$ of $X$
$$
G\text{-}\!\capa^*(E)=0\quad\text{if and only if}\quad c_{\widetilde G}^*(E)=0.
$$
\end{cor}

This follows from Theorem \ref{thm4.3} in view of the usual definition of inner and outer capacities, see e.g.\   \cite[p.\ 38]{Fu4}.

\subsection{Sweeping of measures.} After this excursion to possibly non-symmetric kernels in the preceding subsection we shall from now on again only consider symmetric kernels $G$, strictly positive on the diagonal. We henceforth assume that $G$ is consistent and positive definite. For simplicity of statement we even assume that $G$ is {\textit{strictly}} positive definite, and hence altogether {\textit{perfect}} in the sense of \cite[Definition 3.3 and Theorem 3.3]{Fu1}. Under suitable additional hypotheses we shall then define {\textit{balayage}} (sweeping) of an arbitrary measure $\omega\in\cal M^+$ onto a {\textit{quasiclosed}} set $A\subset X$ such that $f:=1_AG\omega$ has $c^*(f)<+\infty$ and hence $f\in\cal H_0^*$, by Lemma \ref{lemma4.0}. That will make  Theorems \ref{thm2.1} and \ref{thm3.9} applicable. In the first place this will lead to Theorem \ref{thm4.4} about {\textit{pseudobalayage}}, or improper balayage, of $\omega$ on $A$, but in the presence of a suitable domination principle we obtain proper balayage in Theorem {\ref{thm4.5}}.

The Gauss integral associated with $f$ reads
\begin{equation}\label{Gauss}
 2\int_{A}G\omega\,d\mu-\|\mu\|^2
\end{equation}
as a function of $\mu\in\cal E^+$. For such $\mu$, $A$ is $\mu$-measurable by \cite[Corollary 6.1]{Fu4} because $\complement A$ is quasiopen and so $1_{\complement A}$ is of class $\cal G^*$ by \cite[Lemma 2.4]{Fu4}. By variation of the Gauss integral in Theorems \ref{thm2.1} and \ref{thm3.9} we only need to consider measures $\mu\in\cal E^+(A)$, that is, measures $\mu\in\cal E^+$ carried by $A$, because the Gauss integral, defined by (\ref{Gauss}), increases when $\mu$ is replaced by its restriction to $A$. This leads to (\ref{4.2}) in Theorem \ref{thm4.4} below. Similarly concerning the other variational characterization of $\omega^A$, stated in (\ref{4.3}) in Theorem \ref{thm4.4}. The Gauss integral (\ref{Gauss}) now reads
\begin{eqnarray}\label{4.1}
2\int G\omega\,d\mu-\|\mu\|^2=\|\omega\|^2-\|\omega-\mu\|^2,\quad\mu\in\cal E^+(A),
\end{eqnarray}
the latter expression applicable when $\omega\in\cal E^+$, in which case we have $G\omega\in\cal H^*_0$ by Definition \ref{def3.5}, and hence $1_AG\omega\in\cal H^*_0$ by Lemma \ref{lemma3.8}. In the following theorem we achieve the same for any measure $\omega\in\cal M^+$ for suitable $X$ and $G$ (the hypothesis that $G$ be {\textit{strictly}} positive definite being unnecessary for $1_AG\omega$ to be of class $\cal H^*_0$).

\begin{theorem}\label{thm4.4} Let $G$ be perfect, that is, consistent and strictly positive definite. For a given measure $\omega\in\cal M^+$ suppose either that $\omega\in\cal E^+$ or that $X$ is countable at infinity, that $G$ is  strongly regular, and that $G(x,x)=+\infty$ for every point $x\in\cal O_\infty$.
For any quasiclosed set $A\subset X$ such that $c^*(1_AG\omega)<+\infty$ we have $1_AG\omega\in\cal H^*_0$ and
\begin{align}\label{4.2}
[c(1_AG\omega)]^2&=\max_{\mu\in\cal E^+(A)}\,\Bigl(2\int G\omega\,d\mu-\|\mu\|^2\Bigr),\\
[c(1_AG\omega)]^2&=\max\,\Bigl\{\int G\omega\,d\mu:\ \mu\in\cal E^+(A),\ G\mu\le G\omega\text{\ $\mu$-a.e.}\Bigr\}.\label{4.3}
\end{align}
In either case\/ {\rm(\ref{4.2})} or\/ {\rm(\ref{4.3})} there is precisely one maximizing measure $\mu$, the same in both cases. This maximizing measure, denoted by $\omega^A$, is of class $\cal E^+(A)$ and is characterized within $\cal E^+(A)$ by the following properties {\rm{(a)}} and {\rm{(b)}}:
\begin{itemize}
\item[\rm(a)] $G\omega^A\ge G\omega$ {\text{ q.e.\ on\ }}$A$,
\item[\rm(b)] $G\omega^A=G\omega$ {\text{\ $\omega^A$-a.e.}};
\end{itemize}
and $\omega^A$ has the following further properties:
\begin{itemize}
\item[\rm(c)] $\int G\omega\,d\omega^A=\int G\omega^A\,d\omega=\|\omega^A\|^2=[c(1_AG\omega)]^2$,
\item[\rm(d)] $c(1_AG\omega)=\min\,\bigl\{\|\lambda\|:\ \lambda\in\cal E^+,\ G\lambda\ge G\omega\text{\ q.e.\ on\ } A\bigr\}$.
\end{itemize}
\end{theorem}

\begin{proof} We apply Theorems \ref{thm2.1} and \ref{thm3.9} to $f:=1_AG\omega$. By  Remark \ref{remark4.0}, resp.\ Theorem~\ref{thm4.3a} and Lemma \ref{lemma4.0}~(a), we have $f\in\cal H^*_0$. Thus (\ref{4.2}) (hence also (\ref{4.3})) follows right away from (\ref{2.4}) and (\ref{2.5}). The (common) maximizing measure $\omega^A\in{\rm{M}}(f)$ is uniquely determined in view of the strict positive definiteness of $G$ because ${\rm{M}}(f)$ is a subset of the equivalence class $\Lambda(f)$ (Theorem \ref{thm3.9} (d)).  The properties (a), (b), (c),
including the characterization of $\omega^A$ within $\cal E^+(A)$ by (a) and (b), likewise follow immediately from corresponding assertions in Theorems \ref{thm2.1} and \ref{thm3.9}. Finally, property (d) follows from the equality $c(1_AG\omega)=\gamma(1_AG\omega)$ (see Lemma \ref{lemma3.5}~(i)).
\end{proof}

The unique measure $\omega^A$ in Theorem \ref{thm4.4} is said to arise by {\textit{pseudobalayage}} or improper balayage of $\omega$ on $A$. If $\omega\in\cal E^+$ then $\omega^A$ is the nearest-point projection of $\omega$ on $\cal E^+(A)$ according to (\ref{4.1}).

The most important case of pseudobalayage is of course proper {\textit{balayage}} or {\textit{sweeping}}. In order to achieve sweeping of a given measure $\omega$ on quasiclosed sets $A$ such that $c^*(1_AG\omega)<+\infty$ we assume that the perfect kernel $G$ satisfies the following $\omega$-domination principle:

\begin{definition}\label{def4.5} For a given measure $\omega\in\cal M^+$ the kernel $G$ satisfies the {\textit{$\omega$-domination principle}} if for any $\mu\in\cal E^+$ such that $G\mu\le G\omega$ $\mu$-a.e.\ we have $G\mu\le G\omega$ everywhere on $X$. If this holds for every measure $\omega$ of finite energy we say that $G$ satisfies the {\textit{domination principle}}.
\end{definition}

\begin{theorem}\label{thm4.5}
Let $G$ be perfect. For a given measure $\omega\in\cal M^+$ suppose that $G$ satisfies the $\omega$-domination principle. Also suppose that either $\omega\in\cal E^+$ or that $X$ is countable at infinity, that $G$ is strongly regular, and that $G(x,x)=+\infty$ for every point $x\in\cal O_\infty$. For any quasiclosed set $A\subset X$ such that $c^*(1_AG\omega)<+\infty$ we then have $1_AG\omega\in\cal H^*_0$,
and {\rm(\ref{4.2})} and {\rm(\ref{4.3})} from Theorem {\rm\ref{thm4.4}} hold, both with precisely one maximizing measure $\mu$, the same in both cases. This maximizing measure is termed the sweeping of $\omega$ on $A$ and denoted by $\omega^A$. It is characterized within $\cal E^+(A)$ by the property
\begin{itemize}
\item[\rm(a$'$)] $G\omega^A=G\omega\text{\ q.e.\ on\ }A$,
\end{itemize}
and has the following further properties:
\begin{itemize}
\item[\rm(b$'$)] $G\omega^A\le G\omega\text{\ everywhere on\ }X$,
\item[\rm(c)] $\int G\omega\,d\omega^A=\int G\omega^A\,d\omega=\|\omega^A\|^2=[c(1_AG\omega)]^2$,
\item[\rm(d)] $c(1_AG\omega)=\min\,\bigl\{\|\lambda\|:\ \lambda\in\cal E^+,\ G\lambda\ge G\omega\text{\ q.e.\ on\ }A\bigr\}$.
\end{itemize}
\end{theorem}

\begin{proof}
In fact, (a) and (b) from Theorem \ref{thm4.4} imply together (a$'$) and (b$'$) after applying the $\omega$-domination principle to (b). Conversely, (a$'$) implies both (a) (trivially) and (b) (in view of  (\ref{2.2}), (\ref{2.3}), and because $\omega^A\in\cal E^+(A)$). The remaining assertions follow from Theorem \ref{thm4.4}.
\end{proof}

\subsection{Capacitary measure and equilibrium measure on a quasiclosed set.} It is easy to adapt the proofs of Theorem \ref{thm4.4}, resp.\ \ref{thm4.5}, so as to establish corresponding results (Theorem \ref{thm4.6}, resp.\ Remark \ref{remark4.7}) on the capacitary measure, resp.\ the {\textit{equilibrium measure}}, on a quasiclosed set $A\subset X$ of finite outer capacity $c^*(A)$. We simply replace $G\omega$ by $1$ and $\omega^A$ by the capacitary measure, resp.\ the equilibrum measure on $A$, both denoted by $\mu_A$, and therefore replace $c^*(1_AG\omega$) by $c^*(A)$. The use of Lemma \ref{lemma4.0} in the proof of Theorem \ref{thm4.4} is now replaced by the fact that every quasiclosed set $A\subset X$ with $c^*(A)<+\infty$ is quasicompact according to Lemma \ref{lemma3.8} (b). Furthermore, in order to establish Remark \ref{remark4.7}, the $\omega$-domination principle in Theorem \ref{thm4.5}
is now replaced by the (Frostman) maximum principle. Countability of $X$ at infinity and strong regularity for $G$ (serving to establish quasicontinuity of $G\omega$ in case $\|\omega\|=+\infty$) are not needed here because the function $1$ is even continuous.

\begin{theorem}\label{thm4.6} Let
$G$ be perfect.
For any quasiclosed set $A\subset X$ with $c^*(A)<+\infty$ we have $1_A\in\cal H^*_0$ and
\begin{align}\label{4.6}
c(A)^2&=\max_{\mu\in\cal E^+(A)}\,\int(2-G\mu)\,d\mu=\max_{\mu\in\cal E^+(A)}\,(2\mu(X)-\|\mu\|^2),\\
c(A)^2&=\max\,\bigl\{\mu(X):\ \mu\in\cal E^+(A),\ G\mu\le 1\text{\ $\mu$-a.e.}\bigr\}.\label{4.7}
\end{align}
In either case\/ {\rm(\ref{4.6})} or\/ {\rm(\ref{4.7})} there is precisely one maximizing measure $\mu$, the same in both cases. This maximizing measure, termed the capacitary measure on $A$ and denoted by $\mu_A$, is carried by $A$ and is characterized within $\cal E^+(A)$ by the following properties {\rm(a)} and {\rm(b)}:
\begin{itemize}
\item[\rm(a)] $G\mu_A\ge1$ q.e.\ on $A$,
\item[\rm(b)] $G\mu_A=1$ $\mu_A$-a.e.,
\end{itemize}
and $\mu_A$ has the following further properties:
\begin{itemize}
\item[\rm(c)] $\mu_A(A)=\|\mu_A\|^2=c(A)^2$,
\item[\rm(d)] $c(A)=\min\,\bigl\{\|\lambda\|:\ \lambda\in\cal E^+,\ G\lambda\ge 1\text{\ q.e.\ on\ }A\bigr\}$.
\end{itemize}
\end{theorem}

\begin{remark}\label{remark4.7} If in addition $G$ satisfies the {\textit{maximum principle}} then we get the actual equilibrium measure (rather than the capacitary measure). Now (a) of the above theorem becomes (a$'$) $G\mu_A=1$ q.e.\ on $A$,
whereas (b) becomes
(b$'$) $G\mu_A\le1$ everywhere on $X$.
Furthermore, (a$'$) alone now characterizes $\mu_A$ within $\cal E^+(A)$.
\end{remark}

\subsection{Outer balayage on an arbitrary set.}\label{OuterBal} We apply Theorem \ref{thm4.5}
to $f=1_AG\omega$, where $\omega\in\cal M^+$ is a given measure and $A$ now is an arbitrary subset of $X$ with $c^*(1_AG\omega)<+\infty$ (instead of a quasiclosed set with that property). Assuming that $X$ has a countable base and that $G$ is perfect, there is a $c^*$-equivalence class of {\textit{quasiclosures}} $A^*$ of $A$, that is, quasiclosed sets quasicontaining $A$ and quasi minimal with these two properties. Explicitly, let us say that a set $B$ {\textit{quasicontains}} a set $A$ if $c^*(A\setminus B)=0$. Then a quasiclosure of $A$ is defined as a quasiclosed set $A^*$ quasicontaining $A$ and such that every quasiclosed set $B$ which quasicontains $A$ also quasicontains $A^*$. Equivalently, $1_{A^*}$ shall equal $(1_A)^*$, a quasi u.s.c.\ envelope of $1_A$ (see the paragraph preceding Definition \ref{def3.16}). Directly, a quasiclosure $A^*$ of $A$ exists according to \cite[Theorem 2.7]{Fu3} applied to the outer capacity $C$ on (subsets of) $X$ defined by $C(E)=c^*(1_EG\omega)$, $E\subset X$, noting that $C$ is sequentially order continuous from below (on arbitrary sets) because $c^*$ is sequentially order continuous from below on $\cal F^+$ (Corollary \ref{cor3.13}). The $c^*$-equivalence class of all quasiclosures of $A$ obviously depends only on the $c^*$-equivalence class of $A$. From Theorem \ref{thm4.5} we have in view of Lemma \ref{lemma3.8} (cf.\ Theorem \ref{thm4.3a})  the following result on outer balayage:

\begin{theorem}\label{cor4.8} Suppose that $X$ has a countable base and that $G$ is perfect. Consider a measure $\omega\in\cal M^+$ such that $G$ satisfies the $\omega$-domination principle, a set $A\subset X$ with $c^*(1_AG\omega)<+\infty$, and a quasiclosure $A^*$ of $A$. Suppose moreover that either $\omega\in\cal E^+$ or that $G$ is strongly regular and that $G(x,x)=+\infty$ for every $x\in\cal O_\infty$. Then $c^*(1_AG\omega)=c(1_{A^*}G\omega)$. The sweeping $\omega^{A^*}$ of $\omega$ on $A^*$ is also called the {\textit{outer sweeping}} of $\omega$ on $A$ and denoted by $\omega^{*A}$. It is carried by $A^*$ and is characterized within $\cal E^+(A^*)$ by the following property:
\begin{itemize}
\item[\rm(a$^*$)] $G\omega^{*A}=G\omega$ q.e.\ on $A^*$ (hence q.e.\ on $A$);
\end{itemize}
and $\omega^{*A}$ has the following further properties:
\begin{itemize}
\item[\rm(b$^*$)] $G\omega^{*A}\le G\omega$ everywhere on $X$,
\item[\rm(c$^*$)] $\int G\omega\,d\omega^{*A}=\int G\omega^{*A}\,d\omega=\|\omega^{*A}\|^2=[c^*(1_AG\omega)]^2$,
\item[\rm(d$^*$)] $c^*(1_AG\omega)=\min\,\bigl\{\|\lambda\|:\ \lambda\in\cal E^+,\ G\lambda\ge G\omega\text{\ q.e.\ on\ }A\bigr\}$.
\end{itemize}
\end{theorem}

Concerning \rm(d$^*$), since $A^*$ quasicontains $A$ we have
$$
\bigl\{\lambda\in\cal E^+:\ G\lambda\ge G\omega\text{\ q.e.\ on\ }A^*\bigr\}\subset
\bigl\{\lambda\in\cal E^+:\ G\lambda\ge G\omega\text{\
q.e.\ on\ }A\bigr\}.
$$
Here equality prevails. In fact, being of class $\cal H^*_0$ by Definition \ref{def3.5}, $G\lambda$ is quasi u.s.c., and because $G\omega$ is l.s.c., the set $\{G\lambda\ge G\omega\}$ is quasiclosed; and since it quasicontains $A$ it also quasicontains $A^*$ by definition of $A^*$.

\begin{cor}\label{cor4.9}  Under the hypotheses on $X$, $G$, $\omega$, and $A$ in Theorem {\rm\ref{cor4.8}}, suppose in addition that $\omega\in\cal E^+(A^*)$. Then
\begin{itemize}
\item[\rm(a)] $\omega^{*A}=\omega$,
\item[\rm(b)] $(\omega^{*A}){}^{*B}=\omega^{*A}=(\omega^{*B}){}^{*A}$ for any set $B$ quasicontaining $A$.
\end{itemize}
\end{cor}

Here (a) follows from the unique characterization of $\omega^{*A}$ by the property (a$^*$) in Theorem \ref{cor4.8}, obviously possessed by $\omega$ itself. The former equality (b) holds by (a) applied with $\omega$ and $A$ replaced by $\omega^{*A}$ and $B$, respectively, noting that $\omega^{*A}\in\cal E^+(A^*)\subset\cal E^+(B^*)$ since $A^*$ is quasicontained in $B^*$ and since $\omega^{*A}$ does not charge the sets of zero upper capacity. The latter equality (b) holds
by the characterization of $\omega^{*A}$ within $\cal E^+(A^*)$ by the same property (a$^*$), but now with $\omega$ replaced by $\omega^{*B}$, noting that $G\bigl((\omega^{*B}){}^{*A}\bigr)=G(\omega^{*B})=G\omega=G(\omega^{*A})$ q.e.\ on $A^*$ (the second equality valid even q.e.\ on $B^*\supset A^*$), and so indeed $(\omega^{*B}){}^{*A}=\omega^{*A}$.

\subsection{Outer equilibrium on an arbitrary set.} Replacing $G\omega$ in Theorem \ref{cor4.8} by the constant function $1$ we obtain a similar result for outer equilibrium, whereby the $\omega$-domination principle shall be replaced by the maximum principle, and there is no need for strong regularity of $G$ here, the function $1$ being continuous.

\begin{cor}\label{cor4.11} Suppose that $X$ has a countable base and that $G$ is perfect and satisfies the maximum principle. Consider a set $A\subset X$ with $c^*(A)<+\infty$ and a quasiclosure $A^*$ of $A$. Then $c^*(A)=c(A^*)$. The equilibrium measure $\mu_{A^*}$ on $A^*$ is also called the {\textit{outer equilibrium measure}} on $A$ and denoted by $\mu_A^*$.
It is carried by $A^*$ and is characterized within $\cal E^+(A^*)$ by the following property:
\begin{itemize}
\item[\rm(a$^*$)] $G\mu_A^*=1$ q.e.\ on $A^*$ (hence q.e.\ on $A$).
\end{itemize}
Furthermore, $\mu_A^*$ and $c^*(A)$ have the following properties:
\begin{itemize}
\item[\rm(b$^*$)] $G\mu_A^*\le1$ everywhere on $X$,
\item[\rm(c$^*$)] $\mu_A^*(X)=\|\mu_A^*\|^2=c^*(A)^2$,
\item[\rm(d$^*$)] $c^*(A)=\min\,\bigl\{\|\lambda\|:\ \lambda\in\cal E^+,\ G\lambda\ge1\text{\ q.e.\ on\ }A\bigr\}$.
\end{itemize}
\end{cor}

We omit the quite parallel study of the inner balayage and the inner equilibrium for an arbitrary set $A\subset X$. The following theorem will not be used in the present study.

\begin{theorem}\label{thm4.16} Suppose that $X$ has a countable base and that $G$ is perfect. For any quasiclosed set $A\subset X$ the convex set $\cal E^+_\alpha(A)$ is strongly closed in $\cal E_\alpha$.
\end{theorem}

\begin{proof} Consider a sequence $(\mu_j)\subset\cal E^+_\alpha(A)$ which converges strongly and hence vaguely to some measure $\mu\in\cal E^+_\alpha$, cf.\ \cite[Definition 3.3]{Fu1}. The sequence $(\mu_j)$ is bounded, say $\|\mu_j\|\le a$ for some constant $a$, and $\|\mu\|\le\liminf_j\|\mu_j\|\le a<+\infty$, by vague convergence. Furthermore, $\mu$ is carried by the quasiclosed set $A$ according to \cite[Corollary 6.2]{Fu4}, so indeed $\mu\in\cal E^+(A)$.
\end{proof}

\subsection{Quasi topology and fine topology on $X$ with respect to $G$.}
Throughout this subsection we assume that $X$ has a countable base and that $G$ is consistent and positive definite. We have on $X$ the Cartan \textit{fine topology} with respect to the kernel $G$, that is, the coarsest topology for which every potential $G\mu$, $\mu\in\cal E^+$, is continuous.
We proceed to obtain two results which together express the equivalence of ``quasitopological'' properties and corresponding properties relative to the fine topology on $X$. The second and deeper result is based on the following \textit{lower envelope principle}.

\begin{theorem}\label{lower} Suppose that $G$ satisfies the domination principle. For any family $(\mu_\alpha)\subset\cal E^+$ there then exists a measure $\mu\in\cal E^+$ such that
    $$
    G\mu\le\widehat{\inf_\alpha}\,G\mu_\alpha\text{ \;q.e.\qquad and}\qquad\widehat{\inf_\alpha}\,G\mu_\alpha\le\inf_\alpha\,G\mu_\alpha\le G\mu.
    $$
    In particular, $\inf_\alpha\,G\mu_\alpha$ is quasi u.s.c.\ and equal q.e.\ to $\widehat{\inf}_\alpha\,G\mu_\alpha$.
    \end{theorem}

Here and in what follows $\widehat{f}$ stands for the greatest l.s.c.\ minorant of $f\in\cal F^+$.

\begin{proof} Suppose first that the family is a sequence $(\mu_j)$. The function $f:=\inf_jG\mu_j$ is quasi u.s.c.\ along with each $G\mu_j\in\cal H^*_0$, by consistency. Hence, $f\in\cal H^*_0$, cf.\ \cite[Theorem 2.5]{Fu4}. Let $\mu\in\mathrm{M(f)}$ be a capacitary measure for $f$, cf.\ Theorem \ref{thm2.1}. Then $\mu\in\cal E^+$, and by (b) in that theorem $G\mu=f$ $\mu$-a.e., that is $G\mu\le G\mu_j$ $\mu$-a.e.\ for every $j$. By the domination principle, $G\mu\le G\mu_j$ everywhere for every $j$, that is, $G\mu\le f$ and actually $G\mu\le\widehat f$ since $G\mu$ is l.s.c. On the other hand, by (a), $G\mu\ge f\ge\widehat{f}$ q.e. By a lemma of Choquet, see e.g.\ \cite[p.\ 169]{CC}, an arbitrary family $(\mu_\alpha)\subset\cal E^+$ has a countable subfamily $(\mu_{\alpha_j})$ such that $\widehat{\inf}_\alpha\,G\mu_\alpha=\widehat{\inf}_{\alpha_j}\,G\mu_{\alpha_j}$. This reduces the former assertion of the theorem to the above case of a countable family. In particular, $\inf_{\alpha_j}G\mu_{\alpha_j}=G\mu$ q.e., and since $G\mu$ is quasicontinuous, being of class $\cal H^*_0$ by consistency, so is therefore $\inf_{\alpha_j}G\mu_{\alpha_j}$.
\end{proof}

\begin{remark}\label{fundamental} In Newtonian potential theory there is equality in the first inequality $G\mu\le\widehat{\inf_\alpha}\,G\mu_\alpha$ of the above theorem, which therefore becomes the fundamental convergence theorem.
\end{remark}

\begin{remark}\label{kishi}For the case where $X$ is compact and $G>0$ is finite off the diagonal, symmetric, positive definite, and continuous, it was shown by Kishi \cite[Theorems 3.1 and 3.2]{K} that $G$ satisfies the domination principle if and only if $G$ satisfies the (strong) lower envelope principle, that is, for any two measures $\mu\in\cal E^+$ and $\nu\in\cal M^+$ there exists $\lambda\in\cal M^+$ such that $G\lambda=\min\{G\mu,G\nu\}$ n.e. on $X$. (Consistency of $G$ is not required.)
\end{remark}

Let $Y$ denote a locally compact space with a countable base.

\begin{theorem}\label{equiv1} Every quasicontinuous function $f: X\to Y$ is finely continuous q.e. Every quasi u.s.c., resp.\ quasi l.s.c., function $f: X\to[0,+\infty]$ is q.e.\ finely u.s.c., resp.\ finely l.s.c. Every quasiclosed, resp.\ quasiopen, subset of $X$ differs by some set of zero outer capacity from its fine closure, resp.\ from its fine interior.
\end{theorem}

\begin{proof} This follows from \cite[Theorem IV, 3]{Br} because $c^*$ is finely stable in the sense that $c^*(\widetilde A)=c^*(A)$ for any $A\subset X$ (Corollary \ref{cor3.13}), $\widetilde A$ denoting the fine closure of $A$. In fact, $\gamma^*(A)$ depends only on $\Gamma^*(1_A)=\Gamma^*(1_{\widetilde A})$ from Section 3.1, by fine continuity of $G\lambda$ for every $\lambda\in\cal E^+$.
\end{proof}

The above theorem has the following converse which Brelot called the Choquet property in view of \cite{Ch5} (for classical potential theory).

\begin{theorem}\label{equiv2} Every finely continuous function $f: X\to Y$ is quasicontinuous. Every finely u.s.c., resp.\ finely l.s.c., function is quasi u.s.c., resp. quasi l.s.c. Every finely closed, resp.\ finely open, subset of $X$ is
quasiclosed, resp.\ quasiopen.
\end{theorem}

\begin{proof} In view of Theorem \ref{lower} this follows from \cite[Theorem IV.7]{Br}.
\end{proof}

\begin{remark}\label{regular} In the setting of Theorem \ref{cor4.8} the outer balayage measure $\omega^{*A}$ is carried by the fine closure $\widetilde A$ of $A$, and even by the base of $\widetilde A$ (the set of points of $\widetilde A$ at which $A$ is not thin) because the remaining points of $\widetilde A$ form a set of zero outer capacity, and thus a null set for $\omega^{*A}$.
\end{remark}

\begin{remark}\label{riesz} By way of example, Theorems \ref{lower}, \ref{equiv1}, and \ref{equiv2} hold for the M. Riesz kernels $|x-y|^{\alpha-n}$ on $R^n$ of order $0<\alpha\le2$. As shown in \cite[Sections 5.5--5.7]{Fu3}, the dilated domination principle implies a dilated form of Theorem \ref{lower} which suffices for the purpose of establishing Theorem \ref{equiv2} (and \ref{equiv1}) above. That covers the case of the M. Riesz kernels of any order $0<\alpha<n$.
\end{remark}

\subsection{Epilogue.}  If we compare the results on balayage (and equilibrium) with respect to the kernel $G$ obtained in the present article with corresponding results by Cartan in \cite{Ca2} for the Newtonian kernel (or the M.~Riesz kernels), the main difference is our use of quasitopological concepts, whereby the theory of balayage even on quasiclosed sets, or outer balayage on more general sets, is obtained by the Gauss variational method. Our results only cover (outer) balayage of $\omega$ on sets $A$ such that $c^*(1_AG\omega)<+\infty$, whereas more general sets $A$ occur in \cite{Ca2}. However, the case $c^*(1_AG\omega)=+\infty$ in \cite{Ca2} uses balayage of superharmonic functions which are not necessarily potentials, and in the present setting that would require severe restrictions on the kernel $G$, beyond the requirements imposed on $G$ for our Theorems \ref{thm4.5} on balayage and \ref{cor4.8} on outer balayage (and similarly for equilibrium and outer equilibrium).

\end{document}